\documentclass{amsart}
\usepackage{amsmath,amsthm,amssymb}
\usepackage{mathtools}
\usepackage{cases}
\usepackage{enumerate}
\usepackage{hyperref}

\allowdisplaybreaks[3]
\numberwithin{equation}{section}
\DeclarePairedDelimiter{\abs}{\lvert}{\rvert} % | | absolute value
\theoremstyle{plain} 
\newtheorem{theorem}{Theorem}[section]

\newtheorem{lemma}{Lemma}[section]
\newtheorem{cor}{Corollary}[section]

\begin{document}
\title[Apostol-Vu double zeta-function]{Mean value theorems for the Apostol-Vu double zeta-function and its application}
\author[Y.~Toma]{Yuichiro Toma}
\date{} 
\address{Draduate School of Mathematics, Nagoya University \\ Chikusa-ku, Nagoya 464-8602, Japan}
\email{m20034y@math.nagoya-u.ac.jp}

\makeatletter
\@namedef{subjclassname@2020}{\textup{2020} Mathematics Subject Classification}
\makeatother


\maketitle

\begin{abstract}
Mean value theorems for various double zeta-functions have been studied. In this paper, we consider mean square values of the Apostol-Vu type. In addition, we also obtain mean square value of the Mordell-Tornheim type in the region where the mean value formulas have not been obtained before.
\end{abstract}

\section{Introduction and the statement of results}
Let $s_j =\sigma_j+it_j$ ($j=1,2,3$) be complex variables with $\sigma_j, t_j \in \mathbb{R}$. The Apostol-Vu double zeta-function is defined by 
\begin{equation}
\label{av,2}
\zeta_{AV,2} (s_1,s_2,s_3) = \sum_{m=1}^\infty \sum_{n<m} \frac{1}{m^{s_1}n^{s_2}(m+n)^{s_3}}
\end{equation}
in the region $\sigma_j>1$, by Matsumoto \cite{Ma}. It can be continued meromorphically to the whole $\mathbb{C}^3$-space, and its singularities are $s_1+s_3=1-\ell$ or $s_1+s_2+s_3=2-\ell$ for $\ell \in \mathbb{N}_{\geq 0 }$ ({\cite[Theorem 2]{Ma}}). 

The main purpose of this paper is to show mean value theorems for the above double zeta-function. The mean value theorems for the Riemann zeta-function have been studied by many mathematicians with a motivation to study the relation to the Lindel\"{o}f hypothesis. 

In recent years, Matsumoto and Tsumura \cite{MaTsu} first considered the mean value theorems for the Euler-Zagier double zeta-function which is defined by
\[
\zeta_{EZ,2} (s_1,s_2) = \sum_{m=1}^\infty \sum_{n=1}^\infty \frac{1}{m^{s_1}(m+n)^{s_2}}.
\]
In \cite{MaTsu}, they showed some asymptotic formulas for $\int_2^T \abs{ \zeta_{EZ,2} (s_1,s_2) }^2 d t_2$ with a fixed complex number $s_1$ and any large positive number $T>2$. After \cite{MaTsu}, various mean values of the Euler-Zagier double zeta-function have been proved (e.g. \cite{IMaNa}, \cite{KiMi}, \cite{IKiMa1}, \cite{IKiMa2}, \cite{Ki} and \cite{BMT}).

As for other multiple zeta-functions, Okamoto and Onozuka \cite{OO} considered the Mordell-Tornheim type which is defined by 
\begin{equation*}
\label{FE}
\zeta_{MT,2} (s_1,s_2,s_3) = \sum_{m=1}^\infty \sum_{n=1}^\infty \frac{1}{m^{s_1}n^{s_2}(m+n)^{s_3}}.
\end{equation*}
This series is absolutely convergent when $\sigma_1+\sigma_3>1, \sigma_2+\sigma_3>1$ and $\sigma_1+\sigma_2+\sigma_3>2$ ({\cite[Theorem 2.2]{OO}}). In their paper, they defined the function
\begin{equation*}
\zeta_{MT,2}^{[2]} (s_1,s_2,s_3) = \sum_{k=2}^\infty \abs*{ \sum_{m=1}^{k-1} \frac{1}{m^{s_1} (k-m)^{s_2}} }^2 \frac{1}{k^{s_3}}
\end{equation*}
which is absolutely convergent when $2\sigma_1+\sigma_3>1, 2\sigma_2+\sigma_3>1$ and $2\sigma_1+2\sigma_2+\sigma_3>2$ ({\cite[Theorem 2.2]{OO}}), and then they showed that
\[
\int_2^T \abs{ \zeta_{MT,2}(s_1,s_2,s_3)}^2 dt_3 = \zeta_{MT,2}^{[2]} (s_1,s_2,2\sigma_3) T + o(T)
\]
with complex numbers $s_1$ and $s_2$, where ($s_1,s_2,s_3$) are in a domain $\mathcal{D}$ which is defined by
\begin{align*}
\mathcal{D} &= \left\{ (s_1,s_2,s_3) \mid \sigma_1+\sigma_3 >1, \sigma_2+\sigma_3 >1 \text{\ and\ } \sigma_1+\sigma_2+\sigma_3>2 \right\} \\
& \cup \left\{ (s_1,s_2,s_3) \mid \sigma_1>1, \sigma_2 \geq 0, \sigma_3>0, t_2 \geq 0, 1/2<\sigma_2+\sigma_3 \leq 1, \right. \\
&\quad \left. \sigma_1+\sigma_2+\sigma_3>2, s_2+s_3 \neq 1 \text{\ and\ } t_3 \geq 2 \right\} \\
& \cup \left\{ (s_1,s_2,s_3) \mid 1/2 < \sigma_1<3/2, \sigma_2 \geq 0, \sigma_3>0, t_2 \geq 0, \sigma_1+\sigma_3 > 1, \right. \\
&\quad \left. 1/2 <\sigma_2+\sigma_3 \leq 1, 3/2 <\sigma_1+\sigma_2+\sigma_3\leq 2, s_2+s_3 \neq 1, \right. \\
&\quad \left. s_1+s_2+s_3 \neq2 \text{\ and\ } t_3 \geq 2 \right\}.
\end{align*}

In the present paper, we consider the properties of $\zeta_{AV,2}(s_1,s_2,s_3)$ with respect to $s_3$, while $s_1,s_2$ are to be fixed. Also we define the function by
\begin{equation}
\label{av,2^[2]}
\zeta_{AV,2}^{[2]} (s_1,s_2,s_3) = \sum_{k=3}^\infty \abs*{ \sum_{k/2 <m \leq k-1} \frac{1}{m^{s_1}(k-m)^{s_2}} }^2 \frac{1}{k^{s_3}}.
\end{equation}
Under these conditions, we will show the following theorems.

%%%Theorem1
\begin{theorem}
\label{thm:first thm}
Let $s_j=\sigma_j+it_j \in \mathbb{C}$ ($j=1,2,3$) with $\sigma_1+\sigma_3>1$ and $\sigma_1+\sigma_2+\sigma_3>2$. Then we have
\begin{equation*}
\int_2^T \abs*{ \zeta_{AV,2} (s_1,s_2,s_3) }^2 dt_3 = \zeta_{AV,2}^{[2]} (s_1,s_2,2\sigma_3) T + O(1) \quad (T \to \infty),
\end{equation*}
where the implicit constant depends on $s_1,s_2$ and $\sigma_3$.
\end{theorem}

As discussed later in Section \ref{sec2}, we see that the region mentioned in the statement of the theorem is the region of absolute convergence for $\zeta_{AV,2} (s_1,s_2,s_3)$. Although the above formula is easily shown,  we further prove the following mean square formulas for $\zeta_{AV,2} (s_1,s_2,s_3)$ in the meromorphically continued regions. 

%%%Theorem2
\begin{theorem}
\label{thm:second thm}
Let $s_j=\sigma_j+it_j \in \mathbb{C}$ ($j=1,2,3$) with $\sigma_1\geq 0,\sigma_3>0, \frac{1}{2} < \sigma_1+\sigma_3\leq1$, $\sigma_1+\sigma_2+\sigma_3>2, t_1 \geq 0$ and $t_3 \geq 2$. %Assume that when $t_3$ moves from $2$ to $T$, the point $(s_1,s_2,s_3)$ does not encounter the hyperplane $s_1+s_3=1$. 
Then we have
\begin{align*}
&\int_2^T \abs*{ \zeta_{AV,2} (s_1,s_2,s_3) }^2 dt_3 \\
&= \zeta_{AV,2}^{[2]} (s_1,s_2,2\sigma_3) T + \begin{cases}
O ( T^{2-2\sigma_1-2\sigma_3} \log T ) & (\frac{1}{2} < \sigma_1+\sigma_3 \leq \frac{3}{4}) \\
O ( T^\frac{1}{2} ) & (\frac{3}{4} < \sigma_1+\sigma_3 \leq 1), \\
\end{cases}
\end{align*}
where implicit constants depend on $s_1,s_2$ and $\sigma_3$.
\end{theorem}

%%%Theorem3
\begin{theorem}
\label{thm:third thm}
Let $s_j=\sigma_j+it_j \in \mathbb{C}$ ($j=1,2,3$) with $\sigma_1\geq 0, t_1 \geq 0, \sigma_3>0, t_3 \geq 2, \sigma_1+\sigma_3>\frac{1}{2}$ and $2 \geq \sigma_1+\sigma_2+\sigma_3>\frac{3}{2}$. Assume that when $t_3$ moves from $2$ to $T$, the point $(s_1,s_2,s_3)$ does not encounter the hyperplane %$s_1+s_3=1$ and 
$s_1+s_2+s_3=2$. Then we have
\begin{align*}
&\int_2^T \abs*{ \zeta_{AV,2} (s_1,s_2,s_3) }^2 dt_3 \\
&= \zeta_{AV,2}^{[2]} (s_1,s_2,2\sigma_3) T \\
&+ \begin{cases}
O ( T^{2-2\sigma_1-2\sigma_3} \log T ) & (\sigma_2 \geq \frac{1}{2}+\sigma_1+\sigma_3) \\
O ( T^{\frac{5}{2}-\sigma_1-\sigma_2-\sigma_3} ) & (\sigma_2 < \frac{1}{2}+\sigma_1+\sigma_3, \frac{3}{2} < \sigma_1+\sigma_2+\sigma_3<2) \\
O ( (T \log T)^\frac{1}{2} ) & (\sigma_2 < \frac{1}{2}+\sigma_1+\sigma_3, \sigma_1+\sigma_2+\sigma_3=2), \\
\end{cases}
\end{align*}
where implicit constants depend on $s_1,s_2$ and $\sigma_3$.
\end{theorem}

Furthermore, by using Theorem \ref{thm:approximation2} in Section \ref{sec2} and the following simple relation ({\cite[(5.6)]{Ma}}):
\begin{equation}
\label{functional relation}
\zeta_{MT,2} (s_1,s_2,s_3) = 2^{-s_3} \zeta(s_1+s_2+s_3)+\zeta_{AV,2} (s_1,s_2,s_3)+\zeta_{AV,2} (s_2,s_1,s_3),
\end{equation}
we can easily obtain an aproximation formula for the Mordell-Tornheim type. As a result, we can prove a mean square formula for the Mordell-Tornheim double zeta-function in a domain $\mathcal{D}^\prime$ which is defined by 
\begin{align*}
\mathcal{D}^\prime &= \left\{ (s_1,s_2,s_3) \mid \sigma_1 \geq 0, \sigma_2 \geq 0, \sigma_3 > 0, t_1 \geq 0, t_2 \geq 0, 2 \leq t_3 \leq T, \right. \\
&\qquad \left. \sigma_1+\sigma_3 \leq 1, \sigma_2+\sigma_3 \leq 1, \sigma_1+\sigma_2+\sigma_3> 3/2, \right. \\
&\qquad \left. s_1+s_3 \neq 1, s_2+s_3 \neq 1, s_1+s_2+s_3 \neq 2 \right\}.
\end{align*}

Remark that mean square formulas in $\mathcal{D}^\prime$ have not been given in \cite{OO}. 

%%%Theorem4
\begin{theorem}
\label{thm:forth thm}
Let $s_j=\sigma_j+it_j $  ($j=1,2,3$) are in the domain $\mathcal{D}^\prime$. %Assume that when $t_3$ moves from $2$ to $T$, the point $(s_1,s_2,s_3)$ does not encounter the hyperplane $s_1+s_3=1, s_2+s_3=1$ and $s_1+s_2+s_3=2 $. 
Then we have
\begin{align*}
&\int_2^T \abs*{ \zeta_{MT,2} (s_1,s_2,s_3) }^2 dt_3 \\
&= \zeta_{MT,2}^{[2]} (s_1,s_2,2\sigma_3) T + \begin{cases}
O ( T^{\frac{5}{2}-\sigma_1-\sigma_2-\sigma_3} ) & (\frac{3}{2}<\sigma_1+\sigma_2+\sigma_3<2) \\
O ( (T \log T)^\frac{1}{2} ) & (\sigma_1+\sigma_2+\sigma_3=2) \\
\end{cases}
\end{align*}
where implicit constants depend on $s_1,s_2$ and $\sigma_3$.
\end{theorem}

The method of the present parer is based on that of Matsumoto and Tsumura \cite{MaTsu}, and also of Okamoto and Onozuka \cite{OO}. In particular, some details of the proof are analogous to the argument in \cite{OO}, further Theorems \ref{thm:first thm}, \ref{thm:second thm} and \ref{thm:third thm} of this paper correspond to Theorems 3.1, 3.2 and 3.3 of \cite{OO}, respectively. However, the method of \cite{OO} is not applicable to the Apostol-Vu type as it is. Therefore in Section \ref{sec2}, we have to prove Lemma \ref{lem:integral deribvative} in order to apply their method to the Apostol-Vu type. 

Moreover, in the proof of Theorem \ref{thm:forth thm}, we calculate mean square values of $\zeta_{MT,2}$ in a way different from \cite{OO}. We estimate the term $S_1$ (\ref{S_1}) that is different from $D_1$ in \cite{OO}. We adopt the method of Miyagawa \cite{Mi} which is applied to Barnes double zeta-functions.

\section{Preliminaries}\label{sec2}
In this section, we consider the properties of the Apostol-Vu double zeta-function. Firstly, we give the region of absolute convergence for $\zeta_{AV,2} (s_1,s_2,s_3)$ and $\zeta_{AV,2}^{[2]} (s_1,s_2,s_3)$.

\begin{theorem}
\label{thm:series conv}
The series (\ref{av,2}) is absolutely convegent when $\sigma_1+\sigma_3>1$, and $\sigma_1+\sigma_2+\sigma_3>2$.  Also the series (\ref{av,2^[2]})  is absolutely convegent when $2\sigma_1+\sigma_3>1$ and $2\sigma_1+2\sigma_2+\sigma_3>3$.
\end{theorem}
\begin{proof}
Let $s_j =\sigma_j+it_j \in \mathbb{C}$ ($j=1,2,3$). We have
\begin{align*}
\sum_{m=1}^\infty \sum_{n<m} \abs*{ \frac{1}{m^{s_1}n^{s_2}(m+n)^{s_3}} }
&\ll \sum_{m=1}^\infty \frac{1}{m^{\sigma_1}} \sum_{n<m} \frac{1}{n^{\sigma_2}(m+n)^{\sigma_3}} \\
&\ll \sum_{m=1}^\infty \frac{1}{m^{\sigma_1}} \sum_{n<m} \frac{1}{n^{\sigma_2} m^{\sigma_3}} \\
&\ll \sum_{m=1}^\infty \frac{1}{m^{\sigma_1+\sigma_3}} \times
\begin{cases}
1 & (\sigma_2>1) \\
\log m & (\sigma_2=1) \\
m^{1-\sigma_2} & (\sigma_2<1) \\
\end{cases} \\
&\ll \sum_{m=1}^\infty \begin{cases}
m^{-\sigma_1-\sigma_3} & (\sigma_2>1) \\
m^{-\sigma_1-\sigma_3} \log m & (\sigma_2=1) \\
m^{1-\sigma_1-\sigma_2-\sigma_3} & (\sigma_2<1). \\
\end{cases}
\end{align*}
Hence the series (\ref{av,2}) is absolutely convergent when $\sigma_1+\sigma_3>1$ and $\sigma_1+\sigma_2+\sigma_3>2$. Similarly we have
\begin{align*}
\abs*{ \sum_{k/2 <m \leq k-1} \frac{1}{m^{s_1}(k-m)^{s_2}} } &\ll \sum_{k/2 <m \leq k-1} \frac{1}{m^{\sigma_1}(k-m)^{\sigma_2}} \\
&\ll k^{-\sigma_1} \sum_{k/2 <m \leq k-1} \frac{1}{(k-m)^{\sigma_2}} \\
&\ll k^{-\sigma_1} \times
\begin{cases}
1 & (\sigma_2>1) \\
\log k & (\sigma_2=1) \\
k^{1-\sigma_2} & (\sigma_2<1).
\end{cases}
\end{align*}
Therefore 
\begin{align*}
\zeta_{AV,2}^{[2]} (s_1,s_2,s_3) &\ll \sum_{k=3}^\infty
\begin{cases}
k^{-2\sigma_1-\sigma_3} & (\sigma_2>1) \\
k^{-2\sigma_1-\sigma_3} (\log k)^2 & (\sigma_2=1) \\
k^{2-2\sigma_1-2\sigma_2-\sigma_3} & (\sigma_2<1), \\
\end{cases}
\end{align*}
and hence $\zeta_{AV,2}^{[2]} (s_1,s_2,s_3)$ is absolutely convergent when $2\sigma_1+\sigma_3>1$ and $2\sigma_1+2\sigma_2+\sigma_3>3$. This completes the proof of Theorem \ref{thm:series conv}. 
\end{proof}

\subsection{The first approximation formula}
Next, we show the first approximation formula for the Apostol-Vu double zeta-function which is necessary to prove Theorem \ref{thm:second thm}.

\begin{theorem}
\label{thm:approximation1}
Let $s_j=\sigma_j+it_j $ ($j=1,2,3$) be complex variables with $s_1+s_3 \neq 1$ and $\sigma_1 \geq 0$, and let $y \geq x \geq 1$ and $\kappa > 1$. Suppose $\sigma_3 > \max \{0, 2-\sigma_1-\sigma_2 \}$, $\abs{t_3} \leq 2\pi x/\kappa -\abs{t_1}$
and $(t_1,t_3) \in \{ (t_1,t_3) \mid t_1, t_3 \geq 0, (t_1,t_3) \neq (0,0) \} \cup \{ (t_1,t_3) \mid  t_1, t_3 \leq 0, (t_1,t_3) \neq (0,0) \}$. Then we have
\begin{align}
\label{formula:first appoximation}
&\zeta_{AV,2} (s_1,s_2,s_3) =\sum_{n \leq x} \sum_{n<m\leq y} \frac{1}{m^{s_1}n^{s_2}(m+n)^{s_3}} \\
&+\frac{y^{1-s_1}}{s_1+s_3-1} \sum_{n \leq x} \frac{1}{n^{s_2}(y+n)^{s_3}} +\frac{s_3}{s_1+s_3-1} \sum_{n \leq x}  \frac{1}{n^{s_2-1}} \int_y^\infty \frac{du}{u^{s_1}(u+n)^{s_3+1}} \nonumber \\
&+\frac{2^{-s_3}}{s_1+s_3-1} \sum_{n>x} \frac{1}{n^{s_1+s_2+s_3-1}} +\frac{s_3}{s_1+s_3-1} \sum_{n>x} \frac{1}{n^{s_2-1}} \int_n^\infty \frac{du}{u^{s_1}(u+n)^{s_3+1}} \nonumber \\
&+\begin{cases}
O \left( x^{-\sigma_1-\sigma_3} \right) & (\sigma_2>1) \\
O \left( x^{-\sigma_1-\sigma_3} \log x \right) & (\sigma_2=1) \\
O \left( x^{1-\sigma_1-\sigma_2-\sigma_3} \right) & (\sigma_2<1), \\
\end{cases} \nonumber 
\end{align}
where implicit constants depend on $s_1,s_2, \sigma_3$ and $\kappa$.
\end{theorem}

This Theorem is an analogue of approximation formulas for the double zeta-functions of the Euler -Zagier type which was proved by Matsumoto and Tsumura \cite{MaTsu} and of the Mordell-Tornheim type proved by Okamoto and Onozuka \cite{OO}. Further these are also analogues of approximation formula for the Riemann zeta-function which was proved by Hardy and Littlewood. 

%%%exponential sum formula
\begin{lemma}({\cite[Lemma 4.10]{Tit}})
\label{lem:exponential sum}
Let $f(x)$ be a real function with a continuous and steadily decreasing derivative $f^\prime(x)$ in $(a,b)$, and let $f^\prime(b)=\alpha$, $f^\prime(a)=\beta$. Let $g(x)$ be a real positive decreasing function with a continuous derivative $g^\prime(x)$, satisfying that $\abs{g^\prime(x)}$ is steadily decreasing. Then
\begin{align*}
\sum_{a < n \leq b} g(n) e^{2\pi i f(n)} &= \sum_{\substack{\nu \in \mathbb{Z} \\ \alpha-\eta<\nu<\beta+\eta}} \int_a^b g(x) e^{2\pi i(f(x)-\nu x)} dx \\
&\quad+O \left(g(a) \log (\beta-\alpha+2) \right) +O \left( \abs{g^\prime(a)} \right)
\end{align*}
for an arbitrary $\eta \in(0,1)$.  
\end{lemma}

%%%Proof of the first approximation formula
\begin{proof}[Proof of Theorem \ref{thm:approximation1}]
Let $s_j=\sigma_j+it_j \in \mathbb{C}$ ($j=1,2,3$) with $\sigma_1+\sigma_3>1$ and $\sigma_1+\sigma_2+\sigma_3>2$. Then for $M,N \in \mathbb{N}$ with $M \geq N \geq 2$, we have 
\begin{align}
\label{partial sum}
\zeta_{AV,2} (s_1,s_2,s_3) &= \sum_{n=1}^\infty \sum_{m>n} \frac{1}{m^{s_1}n^{s_2}(m+n)^{s_3}} \nonumber \\
&= \left( \sum_{n \leq N} \sum_{n<m \leq M} + \sum_{n \leq N} \sum_{m> M} + \sum_{n=N+1}^\infty \sum_{m>n} \right) \frac{1}{m^{s_1}n^{s_2}(m+n)^{s_3}} \nonumber \\
&= A + B + C, \nonumber 
\end{align}
say. 

Firstly, by the Euler-Maclaurin formula ({\cite[(2.1.2)]{Tit}}), the terms $B$ and $C$ are 
\begin{align*}
B &= \sum_{n \leq N} \frac{1}{n^{s_2}} \left( \sum_{m> M} \frac{1}{m^{s_1}(m+n)^{s_3}} \right) \\
&= \sum_{n \leq N} \frac{1}{n^{s_2}} \int_M^\infty \frac{du}{u^{s_1}(u+n)^{s_3}}-s_1 \sum_{n \leq N} \frac{1}{n^{s_2}} \int_M^\infty \frac{u-[u]-\frac{1}{2}}{u^{s_1+1}(u+n)^{s_3}} du \nonumber \\
&- s_3 \sum_{n \leq N} \frac{1}{n^{s_2}} \int_M^\infty \frac{u-[u]-\frac{1}{2}}{u^{s_1}(u+n)^{s_3+1}} du -\frac{1}{2M^{s_1}} \sum_{n \leq N} \frac{1}{n^{s_2}(M+n)^{s_3}} \nonumber \\
&= B_1-B_2-B_3-B_4,
\intertext{and}
C &= \sum_{n=N+1}^\infty \frac{1}{n^{s_2}} \left( \sum_{m>n} \frac{1}{m^{s_1}(m+n)^{s_3}} \right) \nonumber \\
&= \sum_{n=N+1}^\infty \frac{1}{n^{s_2}} \int_n^\infty \frac{du}{u^{s_1}(u+n)^{s_3}}-s_1 \sum_{n=N+1}^\infty \frac{1}{n^{s_2}} \int_n^\infty \frac{u-[u]-\frac{1}{2}}{u^{s_1+1}(u+n)^{s_3}} du \nonumber \\
&- s_3 \sum_{n=N+1}^\infty \frac{1}{n^{s_2}} \int_n^\infty \frac{u-[u]-\frac{1}{2}}{u^{s_1}(u+n)^{s_3+1}} du -2^{-s_3-1} \sum_{n=N+1}^\infty \frac{1}{n^{s_1+s_2+s_3}} \nonumber \\
&= C_1-C_2-C_3-C_4,
\end{align*}
say, respectively. We estimate $C_2, C_3, C_4$ and $B_2, B_3, B_4$ first and $C_1$ and $B_1$ later. As for $C_2$, since
$u+n \asymp u$ when $u \geq n$, we have 
\begin{align*}
C_2 &\ll \sum_{n=N+1}^\infty \frac{1}{n^{\sigma_2}} \int_n^\infty \frac{1}{u^{\sigma_1+1}(u+n)^{\sigma_3}} du \\
&\ll \sum_{n=N+1}^\infty \frac{1}{n^{\sigma_2}} \int_n^\infty \frac{1}{u^{\sigma_1+\sigma_3+1}} du. 
\end{align*}
Then, the last integral in the above converges absolutely when $\sigma_1+\sigma_3>0$ and $\int_n^\infty u^{-\sigma_1-\sigma_3-1} du = n^{-\sigma_1-\sigma_3}/(\sigma_1+\sigma_3)$, so we have 
\[
\sum_{n=N+1}^\infty \frac{1}{n^{\sigma_2}} \int_n^\infty \frac{\abs{u-[u]-\frac{1}{2}}}{u^{\sigma_1+1}(u+n)^{\sigma_3}} du \ll \sum_{n=N+1}^\infty \frac{1}{n^{\sigma_1+\sigma_2+\sigma_3}}.
\]
The series on the right hand side is absolutely convergent when $\sigma_1+\sigma_2+\sigma_3>1$, thus $C_2$ is analytic when $\sigma_1+\sigma_3>0$ and $\sigma_1+\sigma_2+\sigma_3>1$, and we have $C_2 =O( N^{1-\sigma_1-\sigma_2-\sigma_3})$.

The term $C_3$ is absolutely convergent when $\sigma_1+\sigma_3>0$ and $\sigma_1+\sigma_2+\sigma_3>1$, and in the same region we find that $C_3 = O ((\abs{s_3}/\sigma_3) N^{1-\sigma_1-\sigma_2-\sigma_3})$ 
and $C_4$ is absolutely convergent when $\sigma_1+\sigma_2+\sigma_3>1$. Hence 
we find that
\begin{equation}
\label{C_2-C_3-C_4}
C_2,C_3,C_4 \ll  \left(1+(\abs{s_3}/\sigma_3) \right) N^{1-\sigma_1-\sigma_2-\sigma_3}. 
\end{equation}

By the same argument as in the case of $C_2, C_3$ and $C_4$, we can find that $B_2$, $B_3$ and $B_4$ are absolutely convergent when $\sigma_1+\sigma_3>0$, and in the same region they are bounded by 
\begin{equation}
\label{B_2-B_3-B_4}
 B_2, B_3, B_4 \ll \left(1+(\abs{s_3}/\sigma_3) \right) \begin{cases}
N^{-\sigma_1-\sigma_3} & (\sigma_2>1) \\
N^{-\sigma_1-\sigma_3} \log N & (\sigma_2 =1) \\
N^{1-\sigma_1-\sigma_2-\sigma_3} & (\sigma_2 <1), \\
\end{cases}
\end{equation}
since $M>N$.

Secondly we estimate the term $A$. For $x,y \in \mathbb{R}$ with $x<N, y<M, x \leq y$, we divide $A$ as follows:
\begin{align*}
A &= \left( \sum_{n \leq x} \sum_{n<m \leq y} + \sum_{n \leq x} \sum_{y < m \leq M} + \sum_{x < n \leq N} \sum_{n<m \leq M} \right) \frac{1}{m^{s_1}n^{s_2}(m+n)^{s_3}} \\
&= \sum_{n \leq x} \sum_{n<m \leq y} \frac{1}{m^{s_1}n^{s_2}(m+n)^{s_3}} + A_1 + A_2,
\end{align*}
say. Now we fix $n \in \mathbb{N}$ and define two real functions
\[
f(u) = \frac{t_1}{2\pi} \log u + \frac{t_3}{2\pi} \log (u+n), \quad g(u) = \frac{1}{u^{\sigma_1}(u+n)^{\sigma_3}}
\]
in $(y,M)$. When $(t_1, t_3) \in \{ (t_1,t_3) \mid t_1, t_3 \geq 0, (t_1,t_3) \neq (0,0) \}$, then the function $f^\prime(u)$ is steadily decreasing  in $(y,M)$, and since $u > y \geq x$, we find that
\begin{equation*}
\label{f less than 1}
0< f^\prime (u) = \frac{1}{2\pi} \left( \frac{t_1}{u} +\frac{t_3}{u+n} \right) \leq \frac{1}{2\pi u} \left( t_1+t_3 \right) < \frac{1}{2\pi x} \left( t_1+t_3 \right) \leq 1/\kappa <1.
\end{equation*}
Furthermore, the functions $g(x)$ and its derivative $\abs{g^\prime (x)}$ are decreasing when $\sigma_1 \geq 0$ and $\sigma_3>0$. Hence we can apply Lemma \ref{lem:exponential sum} with $(a, b)=(y,M)$ and $\eta =\alpha + \varepsilon = f^\prime (M)+\varepsilon<1$ for small $\varepsilon>0$. Now we consider the following two possibilities: 
\begin{enumerate}[(I)]
\item {\bf The case $\beta+\eta = f^\prime (y)+f^\prime (M)+\varepsilon<1$:} From Lemma \ref{lem:exponential sum}, we have
\[
\sum_{y < m \leq M} g(m) e^{2\pi i f(m)} = \int_y^M g(u) e^{2\pi i f(u)} du + O \left( \frac{1}{y^{\sigma_1}(y+n)^{\sigma_3}} \right).
\]
The above integral is equal to
\[
\int_y^M \frac{e^{2\pi i \left( \frac{t_1}{2\pi} \log u + \frac{t_3}{2\pi} \log (u+n) \right)}}{u^{\sigma_1}(u+n)^{\sigma_3}} du= \int_y^M \frac{1}{u^{\overline{s_1}}(u+n)^{\overline{s_3}}} du,
\]
so we have 
\[
\sum_{y < m \leq M} g(m) e^{2\pi i f(m)} = \int_y^M \frac{1}{u^{\overline{s_1}}(u+n)^{\overline{s_3}}} du+ O \left( \frac{1}{y^{\sigma_1}(y+n)^{\sigma_3}} \right).
\]
 
\item {\bf The case $\beta+\eta = f^\prime (y)+f^\prime (M)+\varepsilon>1$:} In this case we can estimate similarly as in Hardy and Littlewood \cite{HL}. By Lemma \ref{lem:exponential sum}, we have 
\[
\sum_{y < m \leq M} g(m) e^{2\pi i f(m)} = \sum_{\nu =0}^1 \int_y^M g(u) e^{2\pi i \left(f(u) -\nu u\right)} du + O \left( \frac{1}{y^{\sigma_1}(y+n)^{\sigma_3}} \right).
\]
If $\nu=0$, we can calculate as in (I), so remaining task is to estimate the term comes from $\nu=1$. We put $w(u)=u-f(u)=u-\frac{t_1}{2\pi} \log u - \frac{t_3}{2\pi} \log (u+n)$, then 	
\begin{align*}
\frac{dw}{du} &=1-\frac{1}{2\pi} \left( \frac{t_1}{u}+\frac{t_3}{u+n} \right), \\
\frac{d^2 w}{du^2} &= \frac{1}{2\pi} \left( \frac{t_1}{u^2}+\frac{t_3}{(u+n)^2} \right) >0.
\end{align*}
Hence 
\begin{align*}
\Re{ \left( \int_y^M g(u) e^{2\pi i \left( f(u) -u \right)} du \right)} &=\Re{ \left( \int_y^M \frac{e^{-2\pi i w(u)}}{u^{\sigma_1}(u+n)^{\sigma_3}} du \right)} \\
&= \int_{w(y)}^{w(M)} \frac{1}{u^{\sigma_1}(u+n)^{\sigma_3}} \frac{\cos(2\pi w) dw}{1-\frac{1}{2\pi} \left( \frac{t_1}{u}+\frac{t_3}{u+n} \right)}. 
\end{align*}
Since the function $u^{-\sigma_1}(u+n)^{-\sigma_3} \cdot (1-\frac{1}{2\pi}(\frac{t_1}{u}-\frac{t_3}{u+n}) )^{-1}$ is positive monotonically decreasing, from the second mean value theorem there exists $M_1 \in  [y,M]$ such that the above is equal to
\begin{align*}
& \frac{1}{y^{\sigma_1}(y+n)^{\sigma_3}} \frac{1}{1-\frac{1}{2\pi} \left( \frac{t_1}{y}+\frac{t_3}{y+n} \right)} \int_{w(y)}^{w(M_1)} \cos(2\pi w) dw \\
&\ll y^{-\sigma_1}(y+n)^{-\sigma_3} \abs*{ \int_{w(y)}^{w(M_1)} \cos(2\pi w) dw } \\
&\ll y^{-\sigma_1}(y+n)^{-\sigma_3}. 
\end{align*}
Similarly we have 
\[
\Im{ \left( \int_y^M g(u) e^{2\pi i \left( f(u) -u \right)} du \right)} \ll y^{-\sigma_1}(y+n)^{-\sigma_3}.
\]
\end{enumerate}
Therefore from (I) and (II), we obtain the same result 
\begin{equation}
\label{cc}
\sum_{y < m \leq M} g(m) e^{2\pi i f(m)} = \int_y^M \frac{1}{u^{\overline{s_1}}(u+n)^{\overline{s_3}}} du+ O \left( \frac{1}{y^{\sigma_1}(y+n)^{\sigma_3}} \right).
\end{equation}
Taking the complex conjugates on the both sides, we have 
\begin{align}
\label{error is analytic}
\sum_{y < m \leq M} \frac{1}{m^{s_1}(m+n)^{s_3}} &= \int_y^M \frac{du}{u^{s_1}(u+n)^{s_3}} + O \left( \frac{1}{y^{\sigma_1}(y+n)^{\sigma_3}} \right).
\end{align}
Using integration by parts, we have 
\begin{align*}
\sum_{y < m \leq M} \frac{1}{m^{s_1}(m+n)^{s_3}} &= \frac{M^{-s_1+1}(M+n)^{-s_3}}{1-s_1-s_3}+\frac{y^{-s_1+1}(y+n)^{-s_3}}{s_1+s_3-1} \\
&+\frac{ns_3}{s_1+s_3-1} \int_y^M \frac{du}{u^{s_1}(u+n)^{s_3+1}} + O \left( \frac{1}{y^{\sigma_1}(y+n)^{\sigma_3}} \right). 
\end{align*}

Denote the above error term by $E(s_1,s_3;y,n,M)$. We find that the function $E(s_1,s_3;y,n,M)$ is analytic in $s_3$, and the singularity is only $s_3=-s_1+1$ and satisfies 
\begin{equation}
\label{error term 1}
E(s_1,s_3;y,n,M) = O \left( \frac{1}{y^{\sigma_1}(y+n)^{\sigma_3}} \right)
\end{equation}
uniformly in $M$ in the region $\sigma_1+\sigma_3>0$ and $s_1+s_3\neq 1$.

Substituting (\ref{error is analytic}) into $A_1$, we have 
\begin{equation*}
A_1= \sum_{n \leq x} \frac{1}{n^{s_2}} \int_y^M \frac{du}{u^{s_1} (u+n)^{s_3}} + \sum_{n \leq x} \frac{E(s_1,s_3;y,n,M)}{n^{s_2}}.
\end{equation*}

Similarly, we apply Lemma \ref{lem:exponential sum} to $A_2$ with $(a, b)=(n,M)$, and denote its error term by $E(s_1,s_3;n,M)$, then we have 
\begin{equation*}
A_2 = \sum_{x < n \leq N} \frac{1}{n^{s_2}} \int_n^M \frac{du}{u^{s_1} (u+n)^{s_3}} + \sum_{x < n \leq N} \frac{E(s_1,s_3;n,M)}{n^{s_2}},
\end{equation*}
where the function $E(s_1,s_3;n,M)$ is analytic in $s_3$ except for $s_3=-s_1+1$ and satisfies 
\begin{equation}
\label{error term 2}
E(s_1,s_3;n,M) = O \left( \frac{1}{n^{\sigma_1+\sigma_3}} \right)
\end{equation}
uniformly in $M$ in the region $\sigma_1+\sigma_3>0$ and $s_1+s_3\neq 1$. 

Hence we obtain
\begin{align}
\label{before integration by part}
A_1+A_2+B_1+C_1 &= \sum_{n \leq x} \frac{1}{n^{s_2}} \int_y^M \frac{du}{u^{s_1} (u+n)^{s_3}} +\sum_{x < n \leq N} \frac{1}{n^{s_2}} \int_n^M \frac{du}{u^{s_1} (u+n)^{s_3}} \nonumber \\
&+\sum_{n \leq N} \frac{1}{n^{s_2}} \int_M^\infty \frac{du}{u^{s_1} (u+n)^{s_3}} +\sum_{n=N+1}^\infty \frac{1}{n^{s_2}} \int_n^\infty \frac{du}{u^{s_1} (u+n)^{s_3}} \nonumber \\
&+\sum_{n \leq x} \frac{E(s_1,s_3;y,n,M)}{n^{s_2}}+\sum_{x < n \leq N} \frac{E(s_1,s_3;n,M)}{n^{s_2}} \nonumber \\
&=\sum_{n \leq x} \frac{1}{n^{s_2}} \int_y^\infty \frac{du}{u^{s_1} (u+n)^{s_3}} + \sum_{x < n \leq N} \frac{1}{n^{s_2}} \int_n^\infty \frac{du}{u^{s_1} (u+n)^{s_3}} \nonumber \\
&+\sum_{n=N+1}^\infty \frac{1}{n^{s_2}} \int_n^\infty \frac{du}{u^{s_1} (u+n)^{s_3}} \nonumber \\
&+\sum_{n \leq x} \frac{E(s_1,s_3;y,n,M)}{n^{s_2}}+\sum_{x < n \leq N} \frac{E(s_1,s_3;n,M)}{n^{s_2}} \nonumber \\
&=\sum_{n \leq x} \frac{1}{n^{s_2}} \int_y^\infty \frac{du}{u^{s_1} (u+n)^{s_3}} + \sum_{n>x} \frac{1}{n^{s_2}} \int_n^\infty \frac{du}{u^{s_1} (u+n)^{s_3}} \nonumber \\
&+\sum_{n \leq x} \frac{E(s_1,s_3;y,n,M)}{n^{s_2}}+\sum_{x < n \leq N} \frac{E(s_1,s_3;n,M)}{n^{s_2}}.  
\end{align}

Therefore we have
\begin{align}
\label{integration by part}
\zeta_{AV,2} (s_1,s_2,s_3) &= \sum_{n \leq x} \sum_{n<m \leq y} \frac{1}{m^{s_1}n^{s_2}(m+n)^{s_3}} + A_1 + A_2 \nonumber \\
&+B_1-B_2-B_3-B_4+C_1-C_2-C_3-C_4 \nonumber \\
&= \sum_{n \leq x} \sum_{n<m \leq y} \frac{1}{m^{s_1}n^{s_2}(m+n)^{s_3}} \nonumber \\
&+\sum_{n \leq x} \frac{1}{n^{s_2}} \int_y^\infty \frac{du}{u^{s_1} (u+n)^{s_3}} + \sum_{n>x} \frac{1}{n^{s_2}} \int_n^\infty \frac{du}{u^{s_1} (u+n)^{s_3}} \nonumber \\
&+\sum_{n \leq x} \frac{E(s_1,s_3;y,n,M)}{n^{s_2}}+\sum_{x < n \leq N} \frac{E(s_1,s_3;n,M)}{n^{s_2}} \nonumber \\
&-B_2-B_3-B_4-C_2-C_3-C_4  
\end{align}
in the region $\sigma_1+\sigma_3>1, \sigma_1+\sigma_2+\sigma_3>2, \sigma_1 \geq 0,\sigma_3>0$ and $(t_1,t_3) \in \{ (t_1,t_3) \mid t_1, t_3 \geq 0, (t_1,t_3) \neq (0,0) \}$ (Although the above equality does not give the analytic continuation of $\zeta_{AV,2} (s_1,s_2,s_3)$, we use this identity (\ref{integration by part}) in order to prove Theorem \ref{thm:approximation2}).

Using integration by parts to the first two terms of (\ref{integration by part}), we have
\begin{align*}
&\sum_{n \leq x} \frac{1}{n^{s_2}} \int_y^\infty \frac{du}{u^{s_1} (u+n)^{s_3}} + \sum_{n>x} \frac{1}{n^{s_2}} \int_n^\infty \frac{du}{u^{s_1} (u+n)^{s_3}} \\
&=\frac{y^{1-s_1}}{s_1+s_3-1} \sum_{n \leq x} \frac{1}{n^{s_2}(y+n)^{s_3}} +\frac{s_3}{s_1+s_3-1} \sum_{n \leq x}  \frac{1}{n^{s_2-1}} \int_y^\infty \frac{du}{u^{s_1}(u+n)^{s_3+1}} \nonumber \\
&+\frac{2^{-s_3}}{s_1+s_3-1} \sum_{n>x} \frac{1}{n^{s_1+s_2+s_3-1}} +\frac{s_3}{s_1+s_3-1} \sum_{n>x} \frac{1}{n^{s_2-1}} \int_n^\infty \frac{du}{u^{s_1}(u+n)^{s_3+1}}.
\end{align*}
The second term in the above is convergent absolutely when $\sigma_1+\sigma_3>0$, and the third and the forth terms are convergent absolutely when $\sigma_1+\sigma_3>0$ and $\sigma_1+\sigma_2+\sigma_3>2$. 

Therefore if $\sigma_3 > \max \{0, 2-\sigma_1-\sigma_2 \}$, $s_1+s_3 \neq 1$, $\sigma_1 \geq 0$ and $(t_1,t_3) \in \{ (t_1,t_3) \mid t_1, t_3 \geq 0, (t_1,t_3) \neq (0,0) \}$, then we have
\begin{align*}
&\zeta_{AV,2} (s_1,s_2,s_3) \\
&= \sum_{n \leq x} \sum_{n<m \leq y} \frac{1}{m^{s_1}n^{s_2}(m+n)^{s_3}} \nonumber \\
&+\frac{y^{1-s_1}}{s_1+s_3-1} \sum_{n \leq x} \frac{1}{n^{s_2}(y+n)^{s_3}} +\frac{s_3}{s_1+s_3-1} \sum_{n \leq x}  \frac{1}{n^{s_2-1}} \int_y^\infty \frac{du}{u^{s_1}(u+n)^{s_3+1}} \nonumber \\
&+\frac{2^{-s_3}}{s_1+s_3-1} \sum_{n>x} \frac{1}{n^{s_1+s_2+s_3-1}} +\frac{s_3}{s_1+s_3-1} \sum_{n>x} \frac{1}{n^{s_2-1}} \int_n^\infty \frac{du}{u^{s_1}(u+n)^{s_3+1}} \nonumber \\ 
&+\sum_{n \leq x} \frac{E(s_1,s_3;y,n,M)}{n^{s_2}}+\sum_{x < n \leq N} \frac{E(s_1,s_3;n,M)}{n^{s_2}} \nonumber \\
&-B_2-B_3-B_4-C_2-C_3-C_4.  
\end{align*}
From (\ref{C_2-C_3-C_4}) and (\ref{B_2-B_3-B_4}), we see that $B_2$,$B_3$,$B_4$,$C_2$,$C_3$,$C_4 \to 0 \ (N \to \infty)$. So implicit constants depend on $s_1, s_2$ and $\sigma_3$, but not on $t_3$.

Finally from (\ref{error term 1}) and (\ref{error term 2}), when $\sigma_1+\sigma_3>0$ and $\sigma_1+\sigma_2+\sigma_3>2$ the error terms are 
\begin{align*}
\frac{1}{y^{\sigma_1}} \sum_{n \leq x} \frac{1}{n^{\sigma_2}(y+n)^{\sigma_3}} &\ll \frac{1}{y^{\sigma_1+\sigma_3}} \sum_{n \leq x} \frac{1}{n^{\sigma_2}} \\
&\ll \begin{cases}
x^{-\sigma_1-\sigma_3} & (\sigma_2>1) \\
x^{-\sigma_1-\sigma_3}(\log x) & (\sigma_2=1) \\
x^{1-\sigma_1-\sigma_2-\sigma_3} & (\sigma_2<1), \\
\end{cases} 
\end{align*}
\begin{align*}
\sum_{x < n \leq N} \frac{1}{n^{\sigma_1+\sigma_2+\sigma_3}} &\ll \frac{N^{1-\sigma_1-\sigma_2-\sigma_3}-x^{1-\sigma_1-\sigma_2-\sigma_3}}{1-\sigma_1-\sigma_2-\sigma_3} \\
&\to \frac{x^{1-\sigma_1-\sigma_2-\sigma_3}}{\sigma_1+\sigma_2+\sigma_3-1} \ (N \to \infty).
\end{align*}
Thus we obtain the Theorem \ref{thm:approximation1} in the case $(t_1,t_3) \in \{ (t_1,t_3) \mid t_1, t_3 \geq 0, (t_1,t_3) \neq (0,0) \}$. We can prove the case $(t_1,t_3) \in \{ (t_1,t_3) \mid  t_1, t_3 \leq 0, (t_1,t_3) \neq (0,0) \}$ by a similar argument without considering the complex conjugates on (\ref{cc}).  
\end{proof}

\subsection{The second approximation formula}
Next, we give the second approximation formula for the Apostol-Vu double zeta-function which is necessary to prove Theorem \ref{thm:third thm}.
\begin{theorem}
\label{thm:approximation2}
Let $s_j=\sigma_j+it_j $ ($j=1,2,3$) be complex variables with $\sigma_1 \geq 0, t_1 \geq 0$, $\sigma_3> \max \{ 0, \frac{1}{2}-\sigma_1, \frac{3}{2}-\sigma_1-\sigma_2 \}$ and $t_3 \geq2$. Then when $s_1+s_3 \neq 1, s_1+s_2+s_3 \neq 2$, we have 
\begin{equation}
\label{formula:second appoximation}
\zeta_{AV,2} (s_1,s_2,s_3) = \sum_{m \leq at_3} \sum_{n<m} \frac{1}{m^{s_1}n^{s_2}(m+n)^{s_3}} +\begin{cases}
O ( t_3^{\frac{1}{2}-\sigma_1-\sigma_3} ) & (\sigma_2>\frac{3}{2}) \\
O (t_3^{\frac{1}{2}-\sigma_1-\sigma_3} \log t_3 ) & (\sigma_2=\frac{3}{2}) \\
O ( t_3^{\frac{3}{2}-\sigma_1-\sigma_2-\sigma_3} ) & (\sigma_2<\frac{3}{2}), \\
\end{cases}
\end{equation}
where $a=\max \{1, t_1 \}$ and implicit constants depend on $s_1,s_2$ and $\sigma_3$.
\end{theorem}
In order to prove Theorem \ref{thm:approximation2}, we use the classical Mellin-Barnes integral formula, that is
\[
(1+\lambda)^{-s} = \frac{1}{2 \pi i} \int_{(c)} \frac{\Gamma(s+z)\Gamma(-z)}{\Gamma(s)} \lambda^z dz,
\]
where $s, \lambda \in \mathbb{C}$ with $\sigma =\Re (s) >0, \abs{\arg \lambda}< \pi, \lambda \neq 0, c$ is real with $-\sigma<c<0$, and the path $(c)$ of integration is the vertical line $\Re (z) =c$ (see ({\cite[Section 14.51, p. 289, Corollary]{WW}}).

%%%
Further, we give the following Lemma to prove Theorem \ref{thm:approximation2}.
\begin{lemma}
\label{lem:integral deribvative}
Let $[a,b)$ and $V$ be an interval in $\mathbb{R}$. Assume that the continuous function $f:[a,b) \times V \to \mathbb{R}$ satisfies the following conditions:
\begin{enumerate}[(a)]
\item For any $v \in V$, the integral $\int_a^b f(u,v) du$ converges.  
\item The partial derivative $\frac{\partial f}{\partial v} (u,v)$ is continuous in $[a,b) \times V$.  
\item The integral $\int_a^b \frac{\partial f}{\partial v} (u,v) du$ converges uniformly on any compact sets in $V$.
\end{enumerate}
Then we have the following statements:
\begin{enumerate}[(1)]
\item The function $F_1(v)=\int_a^b f (u,v) du$ has the continuous derivative in $V$, and it holds that
\[
F_1^\prime (v) = \int_a^b \frac{\partial f}{\partial v} (u,v) du.
\] 
\item Moreover if $V \subset [a,b)$, The function $F_2(v)=\int_v^b f (u,v) du$ has the continuous derivative in $V$, and it holds that
\[
F_2^\prime (v) = \int_v^b \frac{\partial f}{\partial v} (u,v) du - f(v,v).
\]
\end{enumerate}
\end{lemma}
\begin{proof}
We note that (1) is classically known (for example {\cite[Chapter I\hspace{-.15em}V, Theorem 14.4]{Su}}). However, we give a full proof since the author could not find another suitable reference.

Take $G(v)=\int_a^b \frac{\partial f}{\partial v} (u,v) du$ and take any $t \in [a,b)$. Then by ($b$), the function
\[
G_t (v) = \int_a^t \frac{\partial f}{\partial v} (u,v) du
\]
is continuous on $V$. By ($c$), we see that $G(v)$ is continuous in $V$ since $\{ G_t\}_{t \in [a,b)}$ converges uniformly to $G(v)$ on any compact subset in $V$. 

Moreover if $[c,x] \subset V$ is an arbitrary bounded closed interval, then
\begin{equation*}
\int_c^x G_t (v) dv = \int_a^t \left( \int_c^x \frac{\partial f}{\partial v} (u,v) dv \right) du
\end{equation*}
holds for any $t \in [a,b)$. Since $\{ G_t\}_{t \in [a,b)}$ converges to $G$ uniformly on $[c,x]$, by the termwise integration, the left hand side in the above converges to $\int_c^x G(v) dv$ as $t \to b-0$. Then 
\begin{equation*}
\int_c^x G(v) dv = \int_a^b \left( \int_c^x  \frac{\partial f}{\partial v} (u,v) dv \right) du
\end{equation*}
holds. The above is equal to
\[
\int_a^b ( f(u,x)-f(u,c) ) du = F_1(x) -F_1(c).
\]
Therefore $F_1$ is differentiable and $(1)$ holds since $x \in V$ is arbitrary.

(2) Take 
\[
H(v) = \int_v^b \frac{\partial f}{\partial v} (u,v) du.
\]
Then by the same argument in $(1)$, we see that the function $H(v)$ is continuous in $V$ and 
\begin{align*}
\int_c^x H(v) dv &= \int_c^x \int_v^b \frac{\partial f}{\partial v} (u,v) du dv \\
&= \int_c^x \int_c^u \frac{\partial f}{\partial v} (u,v) dv du + \int_x^b \int_c^x \frac{\partial f}{\partial v} (u,v) dv du. 
\end{align*}
The second identity holds since $V \subset [a,b)$. Similar to $(1)$, each term in the above is equal to
\begin{align*}
& \int_c^x (f(u,u) -f(u,c)) du +\int_c^b (f(u,x)-f(u,c)) du \\
&=\int_c^x f(u,u) du +\int_c^b f(u,x) du -\int_c^b f(u,c) du \\
&= \int_c^x f(u,u) du +F_2 (x) -F_2 (c).
\end{align*}
Differentiating both sides with respect to $x$, we have
\[
H(x)= f(x,x) + F_2^\prime (x).
\]
Thus we obtain $(2)$ .  
\end{proof}

%%%Proof of second approximation formula
\begin{proof}[Proof of Theorem \ref{thm:approximation2}]
Assume that $\sigma_1 \geq 0, t_1 \geq 0, \sigma_3 > 0, t_3 \geq 2, \sigma_1+\sigma_3>1$ and $\sigma_1+\sigma_2+\sigma_3>2$. Let $a=\max \{1, t_1 \}, x =y= at_3$ and $\kappa=\frac{4\pi}{3}$ in Theorem \ref{thm:approximation1}. Then it satisfies the assumption of Theorem \ref{thm:approximation1}:
\[
\abs{t_3} \leq 2\pi x/\kappa -\abs{t_1}.
\]
Indeed, if $a=\max \{1, t_1 \}, x =y= at_3$ and $\kappa=\frac{4\pi}{3}$ then
\[
\frac{2\pi x}{\kappa} - t_1 =\frac{3}{2} \max \{ 1, t_1 \} t_3-t_1 = \begin{cases}
\frac{3}{2} t_3-t_1 & (t_1 \leq 1) \\ 
\frac{3}{2}t_1 t_3-t_1 & (t_1 > 1). \\ 
\end{cases}
\]
If $t_1 \leq1$, then $\frac{3}{2} t_3-t_1 \geq \frac{3}{2} t_3-1 \geq t_3$ since $t_3 \geq 2$. The case $t_1 >1$ is also true since the equation $\frac{3}{2} t_1t_3-t_1 = t_3$ has no solution if $t_3 \geq 2$.  

Hence let us take $x =y= at_3$ in Theorem \ref{thm:approximation1}. Then by (\ref{integration by part}), if $\sigma_1+\sigma_3>1, \sigma_1+\sigma_2+\sigma_3>2, \sigma_1 \geq 0,\sigma_3>0$ and $(t_1,t_3) \in \{(t_1,t_3) \vert t_1, t_3 \geq 0, (t_1,t_3) \neq 0 \},$ we see that
\begin{align}
\label{AV_2}
&\zeta_{AV,2} (s_1,s_2,s_3) \nonumber \\
&= \sum_{m \leq at_3} \sum_{n<m} \frac{1}{m^{s_1}n^{s_2}(m+n)^{s_3}} \nonumber \\
&+\sum_{n \leq at_3} \frac{1}{n^{s_2}} \int_{at_3}^\infty \frac{du}{u^{s_1} (u+n)^{s_3}} + \sum_{n>at_3} \frac{1}{n^{s_2}} \int_n^\infty \frac{du}{u^{s_1} (u+n)^{s_3}} \nonumber \\
&+\sum_{n \leq at_3} \frac{E(s_1,s_3;at_3,n,M)}{n^{s_2}}+\sum_{at_3 < n \leq N} \frac{E(s_1,s_3;n,M)}{n^{s_2}} \nonumber \\
&-B_2-B_3-B_4-C_2-C_3-C_4. 
\end{align}
Last eight terms are analytic when $\sigma_1+\sigma_2+\sigma_3>1, \sigma_1+\sigma_3>0$ and $s_1+s_3 \neq1$ from (\ref{C_2-C_3-C_4}), (\ref{B_2-B_3-B_4}), (\ref{error term 1}) and (\ref{error term 2}). On the other hand, the second term and the third term are absolutely convergent in $\sigma_1+\sigma_3>1$ and $\sigma_1+\sigma_2+\sigma_3>2$. In order to complete the proof of the Theorem, we consider these terms. Denote the second term and the third term in (\ref{AV_2}) by $D_1$ and $D_2$, respectively.

Now we define two functions
\[
\phi_1(v) = \frac{1}{v^{s_2}} \int_{at_3}^\infty  \frac{1}{u^{s_1} (u+v)^{s_3}} du , \quad \phi_2(v) = \frac{1}{v^{s_2}} \int_v^\infty  \frac{1}{u^{s_1} (u+v)^{s_3}} du.
\]
Then the function $u^{-s_1} v^{-s_2} (u+n)^{-s_3}$ satisfies assumptions of Lemma \ref{lem:integral deribvative}. Hence $\phi_1$ and $\phi_2$ have the continuous derivation in $v$ and 
\begin{align*}
\phi_1^\prime (v) &= \frac{-s_2}{v^{s_2+1}} \int_{at_3}^\infty \frac{1}{u^{s_1} (u+v)^{s_3}} du - \frac{1}{v^{s_2}} \int_{at_3}^\infty  \frac{s_3}{u^{s_1} (u+v)^{s_3+1}} du, \\
\phi_2^\prime (v) &= \frac{-s_2}{v^{s_2+1}} \int_v^\infty \frac{1}{u^{s_1} (u+v)^{s_3}} du + \frac{1}{v^{s_2}} \left( \int_v^\infty  \frac{-s_3}{u^{s_1} (u+v)^{s_3+1}} du -\frac{2^{-s_3}}{v^{s_1+s_3}} \right) \\
&= -\frac{s_2}{v^{s_2+1}} \int_v^\infty \frac{1}{u^{s_1} (u+v)^{s_3}} du  - \frac{1}{v^{s_2}} \int_v^\infty  \frac{s_3}{u^{s_1} (u+v)^{s_3+1}} du -\frac{2^{-s_3}}{v^{s_1+s_2+s_3}}.
\end{align*}

By the Euler-Maclaurin formula ({\cite[(2.1.2)]{Tit}}) again, we have 
\begin{align*}
D_1 &= \int_1^{at_3} \frac{1}{v^{s_2}} \int_{at_3}^\infty \frac{1}{u^{s_1} (u+v)^{s_3}} dudv -s_2 \int_1^{at_3} \frac{v-[v]-\frac{1}{2}}{v^{s_2+1}} \int_{at_3}^\infty \frac{1}{u^{s_1} (u+v)^{s_3}} dudv \\
&- s_3 \int_1^{at_3} \frac{v-[v]-\frac{1}{2}}{v^{s_2}} \int_{at_3}^\infty \frac{1}{u^{s_1} (u+v)^{s_3+1}} dudv +\frac{1}{2} \int_{at_3}^\infty \frac{1}{u^{s_1} (u+1)^{s_3}} du \\
&- \frac{\left(at_3-[at_3]-\frac{1}{2} \right)}{(at_3)^{s_2}} \int_{at_3}^\infty \frac{1}{u^{s_1} (u+at_3)^{s_3}} du \\
&= D_{11}-D_{12}-D_{13}+\frac{1}{2}D_{14}-D_{15}, \\
\intertext{and}
D_2 &= \int_{at_3}^\infty \frac{1}{v^{s_2}} \int_v^\infty \frac{1}{u^{s_1} (u+v)^{s_3}} dudv -s_2 \int_{at_3}^\infty \frac{v-[v]-\frac{1}{2}}{v^{s_2+1}} \int_v^\infty \frac{1}{u^{s_1} (u+v)^{s_3}} dudv \\
&- s_3 \int_{at_3}^\infty \frac{v-[v]-\frac{1}{2}}{v^{s_2}} \int_v^\infty \frac{1}{u^{s_1} (u+v)^{s_3+1}} dudv -\frac{1}{2^{s_3}} \int_{at_3}^\infty \frac{v-[v]-\frac{1}{2}}{v^{s_1+s_2+s_3}} dv \\
&+ \frac{\left(at_3-[at_3]-\frac{1}{2} \right)}{(at_3)^{s_2}} \int_{at_3}^\infty \frac{1}{u^{s_1} (u+at_3)^{s_3}} du \\
&=D_{21}-D_{22}-D_{23}-\frac{1}{2^{s_3}}D_{24}+D_{25},
\end{align*}
say, respectively. Since the term $D_{15}$ is cancelled with $D_{25}$, we have
\begin{equation}
\label{D_1+D_2}
D_1+D_2 = D_{11}-D_{12}-D_{13}+\frac{1}{2}D_{14} + (D_{21}-D_{22}-D_{23}-\frac{1}{2^{s_3}}D_{24}).
\end{equation}

We estimate $D_2$ first and $D_1$ later. The term $D_{24}$ is absolutely convergent in the region $\sigma_1+\sigma_2+\sigma_3>1$ and in this region we have
\begin{equation}
\label{D_24}
D_{24} \ll t_3^{1-\sigma_1-\sigma_2-\sigma_3}.
\end{equation}

As for the term $D_{22}$, by using integration by parts, we have
\begin{align*}
D_{22} &= \frac{s_2 2^{-s_3}}{s_1+s_3-1} \int_{at_3}^\infty \frac{v-[v]-\frac{1}{2}}{v^{s_1+s_2+s_3}} dv \\
&+ \frac{s_2 s_3}{s_1+s_3-1} \int_{at_3}^\infty \frac{v-[v]-\frac{1}{2}}{v^{s_2}} \int_v^\infty \frac{1}{u^{s_1} (u+v)^{s_3+1}} dudv \nonumber \\
&= D_{221}+D_{222},
\end{align*}
say. The term $D_{221}$ is analytic in the region $\sigma_1+\sigma_2+\sigma_3>1$ and $s_1+s_3 \neq 1$. Also, the  term $D_{222}$ satisfies 
\begin{align*}
\int_{at_3}^\infty \frac{\abs{v-[v]-\frac{1}{2} }}{v^{\sigma_2}}\int_v^\infty \frac{1}{u^{\sigma_1}(u+v)^{\sigma_3+1}} dudv &\ll \int_{at_3}^\infty \frac{1}{v^{\sigma_2}}\int_v^\infty \frac{1}{u^{\sigma_1}(u+v)^{\sigma_3+1}} dudv \\
&\ll \int_{at_3}^\infty \frac{1}{v^{\sigma_1+\sigma_2+\sigma_3}} dv. 
\end{align*}
Therefore $D_{222}$ is analytic in the region $\sigma_1+\sigma_3>0, \sigma_1+\sigma_2+\sigma_3>1$ and $s_1+s_3 \neq 1$, and we obtain 
\begin{align}
\label{D_22}
D_{22} &\ll t_3^{-1} t_3^{1-\sigma_1-\sigma_2-\sigma_3} + t_3^{1-\sigma_1-\sigma_2-\sigma_3}. \end{align}

Next, we estimate the term $D_{23}$. By the Mellin-Barnes integral formula, we have
\begin{align*}
D_{23} &= s_3 \int_{at_3}^\infty \frac{v-[v]-\frac{1}{2}}{v^{s_2}} \int_v^\infty \frac{1}{u^{s_1+s_3+1} (1+\frac{v}{u})^{s_3+1}} dudv \\
&= \frac{s_3}{2 \pi i \Gamma(s_3+1)} \int_{at_3}^\infty \int_v^\infty \int_{(c)} \frac{(v-[v]-\frac{1}{2}) \Gamma(s_3+1+z) \Gamma(-z)}{v^{s_2-z}u^{s_1+s_3+1+z}} dzdudv
\end{align*}
where $\max \{-\sigma_1-\sigma_3, -\sigma_3-1 \}<c<0$. Here we can change the order of the integration. Since $\sigma_1+\sigma_3+c>0$ and $\sigma_1+\sigma_2+\sigma_3>2$, the above integral is absolutely convergent with respect to $u$ and $v$, respectively. Moreover using the Stirling formula, we find that this is absolutely convergent with respect to $z$. Hence we obtain
\[
D_{23} = \frac{s_3}{2 \pi i \Gamma(s_3+1)} \left( \int_{at_3}^\infty \frac{v-[v]-\frac{1}{2}}{v^{s_1+s_2+s_3}} dv \right) \int_{(c)} \frac{\Gamma(s_3+1+z) \Gamma(-z)}{s_1+s_3+z} dz.
\]
Then, shift the path $(c)$ to $\left( \frac{1}{2} \right)$. Since the relevant pole is only $z=0$ we have
\begin{align}
\label{D23}
D_{23} &= \frac{s_3}{s_1+s_3} \int_{at_3}^\infty \frac{v-[v]-\frac{1}{2}}{v^{s_1+s_2+s_3}} dv \\
&+ \frac{s_3}{2 \pi i \Gamma(s_3+1)} \left( \int_{at_3}^\infty \frac{v-[v]-\frac{1}{2}}{v^{s_1+s_2+s_3}} dv \right) \int_{(\frac{1}{2})} \frac{\Gamma(s_3+1+z) \Gamma(-z)}{s_1+s_3+z} dz. \nonumber 
\end{align}
The integral $\int_{at_3}^\infty (v-[v]-\frac{1}{2})/v^{s_1+s_2+s_3} dv$ is absolutely convergent in the region $\sigma_1+\sigma_2+\sigma_3>1$, and the last integral is analytic when $\sigma_1+\sigma_3>-\frac{1}{2}$ and $\sigma_3>-\frac{3}{2}$. Applying the Stirling formula and Lemma 4 of \cite{Ma2}, we see that
\begin{align*}
&\frac{1}{\abs{\Gamma(s_3+1)}} \int_{(\frac{1}{2})} \abs*{ \frac{\Gamma(s_3+1+z) \Gamma(-z)}{s_1+s_3+z} } \abs{dz} \\
&= \frac{1}{\abs{\Gamma(s_3+1)}} \int_{-\infty}^\infty \frac{\abs{ \Gamma(\sigma_3+\frac{3}{2} + i(t_3+w)) \Gamma(-\frac{1}{2}-iw)}}{\abs{\sigma_1+\sigma_3+\frac{1}{2}+i(t_1+t_3+w)}} dw \\
&\ll t_3^{-\sigma_3-\frac{1}{2}} + t_3^{-\frac{1}{2}}.
\end{align*}
Therefore we have
\begin{align}
\label{D_23}
D_{23} &\ll t_3^{1-\sigma_1-\sigma_2-\sigma_3} + \abs{s_3} t_3^{1-\sigma_1-\sigma_2-\sigma_3} \left( t_3^{-\sigma_3-\frac{1}{2}} + t_3^{-\frac{1}{2}} \right) \nonumber \\
&\ll t_3^{\frac{3}{2}-\sigma_1-\sigma_2-\sigma_3} \left( t_3^{-\sigma_3} + 1 \right) 
\end{align}
in the region $\sigma_1+\sigma_2+\sigma_3>1, \sigma_1+\sigma_3>-\frac{1}{2}, \sigma_3>-\frac{3}{2}$ and $s_1+s_3 \neq 0$. 

We can estimate the term $D_{21}$ by the same method as in the case of $D_{23}$. Let $c \in \mathbb{R}$ with $\max \{1-\sigma_1-\sigma_3, -\sigma_3 \}<c<0$. Then we have
\begin{align*}
D_{21} &= \frac{1}{2 \pi i \Gamma(s_3)} \int_{at_3}^\infty \int_v^\infty \int_{(c)} \frac{\Gamma(s_3+z) \Gamma(-z)}{v^{s_2-z}u^{s_1+s_3+z}} dzdudv \\
&= \frac{(at_3)^{2-s_1-s_2-s_3}}{2 \pi i (s_1+s_2+s_3-2) \Gamma(s_3)} \int_{(c)} \frac{\Gamma(s_3+z) \Gamma(-z)}{s_1+s_3-1+z} dz.
\end{align*}
Shifting the path from $(c)$ to $\left( \frac{1}{2} \right)$, we have 
\begin{align*}
D_{21} &= \frac{(at_3)^{2-s_1-s_2-s_3}}{(s_1+s_2+s_3-2)(s_1+s_3-1)} \\
&\quad+\frac{(at_3)^{2-s_1-s_2-s_3}}{2 \pi i (s_1+s_2+s_3-2) \Gamma(s_3)} \int_{(\frac{1}{2})} \frac{\Gamma(s_3+z) \Gamma(-z)}{s_1+s_3-1+z} dz.
\end{align*}
Hence the term $D_{21}$ can be meromorphically continued to the region $\sigma_1+\sigma_3>\frac{1}{2}, \sigma_3>-\frac{1}{2}, s_1+s_3\neq 1$ and $s_1+s_2+s_3 \neq 2$, and from the Stirling formula and Lemma 4 of \cite{Ma2} we have
\begin{align}
\label{D_21}
D_{21} &\ll \frac{t_3^{2-\sigma_1-\sigma_2-\sigma_3}}{t_3^2} + \frac{t_3^{2-\sigma_1-\sigma_2-\sigma_3}}{t_3} ( t_3^{\frac{1}{2}-\sigma_3} + t_3^{-\frac{1}{2}} ) \nonumber \\
&\ll t_3^{\frac{3}{2}-\sigma_1-\sigma_2-\sigma_3} \left( t_3^{-\sigma_3} + t_3^{-1} \right). 
\end{align}

Summarizing the above results, we find that $D_2$ is analytic in the region $\sigma_1+\sigma_3>\frac{1}{2}, \sigma_1+\sigma_2+\sigma_3>1, \sigma_3>-\frac{1}{2}, s_1+s_3 \neq 1$ and $s_1+s_2+s_3 \neq 2$. From (\ref{D_24}), (\ref{D_22}), (\ref{D_23}) and (\ref{D_21}), we obtain
\begin{equation}
\label{D_2}
D_2 \ll t_3^{\frac{3}{2}-\sigma_1-\sigma_2-\sigma_3} \left( t_3^{-\sigma_3} + 1 \right).
\end{equation}

Next, we estimate $D_{1}$. First, using integration by parts, we have
\begin{align*}
D_{14} &= \frac{(at_3)^{1-s_1}(at_3+1)^{-s_3}}{s_1+s_3-1} + \frac{s_3}{s_1+s_3-1}\int_{at_3}^\infty \frac{1}{u^{s_1}(u+1)^{s_3+1}} du.
\end{align*}
Since the last integral is absolutely convergent when $\sigma_1+\sigma_3>0$, the term $D_{14}$ is analytic $\sigma_1+\sigma_3>0$ and $s_1+s_3 \neq 1$ and we have 
\begin{equation}
\label{D_14}
D_{14} \ll \frac{t_3^{1-\sigma_1-\sigma_3}}{t_3} +t_3^{-\sigma_1-\sigma_3} \ll t_3^{-\sigma_1-\sigma_3}.
\end{equation}

Next, using integration by parts, we have
\begin{align*}
D_{12} &= \frac{s_2 (at_3)^{1-s_1}}{s_1+s_3-1} \int_1^{at_3} \frac{v-[v]-\frac{1}{2}}{v^{s_2+1}(at_3+v)^{s_3}} dv \\
&+ \frac{s_2 s_3}{s_1+s_3-1} \int_1^{at_3} \frac{v-[v]-\frac{1}{2}}{v^{s_2}} \int_{at_3}^\infty \frac{1}{u^{s_1} (u+v)^{s_3+1}} dudv \\
&=D_{121}+D_{122},
\end{align*}
say. Then we have
\begin{equation}
\label{D_121}
D_{121} \ll t_3^{-\sigma_1-\sigma} \begin{cases}
1 & (\sigma_2>0) \\
\log t_3 & (\sigma_2=0) \\
t_3^{1-\sigma_2} & (\sigma_2<0), \\
\end{cases}
\end{equation}
when $s_1+s_3 \neq 1$. The term $D_{122}$ is absolutely convergent when $\sigma_1+\sigma_3>0$ and $s_1+s_3 \neq 1$, and in the same region we have 
\begin{equation}
\label{D_122}
D_{122} \ll \begin{cases}
t_3^{-\sigma_1-\sigma_3} & (\sigma_2>1) \\
t_3^{-\sigma_1-\sigma_3} \log t_3 & (\sigma_2=1) \\
t_3^{1-\sigma_1-\sigma_2-\sigma_3} & (\sigma_2<1). \\
\end{cases}
\end{equation}

Next we consider the term $D_{13}$. Let $c \in \mathbb{R}$ with $\max \{-\sigma_1-\sigma_3, -\sigma_3-1 \}<c<0$. Then we have
\begin{align*}
D_{13} &= s_3 \int_1^{at_3} \frac{v-[v]-\frac{1}{2}}{v^{s_2}} \int_{at_3}^\infty \frac{1}{u^{s_1} (u+v)^{s_3+1}} dudv \\
&= \frac{s_3}{2 \pi i \Gamma(s_3+1)} \int_1^{at_3} \int_{at_3}^\infty \int_{(c)} \frac{(v-[v]-\frac{1}{2}) \Gamma(s_3+1+z) \Gamma(-z)}{v^{s_2-z}u^{s_1+s_3+1+z}} dzdudv \\
&= \frac{s_3}{2 \pi i \Gamma(s_3+1)} \int_{(c)} \frac{\Gamma(s_3+1+z) \Gamma(-z)}{(at_3)^{s_1+s_3+z}(s_1+s_3+z)} \left( \int_1^{at_3} \frac{v-[v]-\frac{1}{2}}{v^{s_2-z}} dv \right) dz.
\end{align*}
Shifting the path $(c)$ to $\left( \frac{1}{2} \right)$, we have  
\begin{align*}
D_{13} &= \frac{s_3}{(at_3)^{s_1+s_3}(s_1+s_3)} \int_1^{at_3} \frac{v-[v]-\frac{1}{2}}{v^{s_2}} dv \\
&\quad+ \frac{s_3}{2 \pi i \Gamma(s_3+1)} \int_{(\frac{1}{2})} \frac{\Gamma(s_3+1+z) \Gamma(-z)}{(at_3)^{s_1+s_3+z}(s_1+s_3+z)} \left( \int_1^{at_3} \frac{v-[v]-\frac{1}{2}}{v^{s_2-z}} dv \right) dz \\
&= D_{131}+D_{132},
\end{align*}
say. Then we have
\begin{equation}
\label{D_131}
D_{131} \ll t_3^{-\sigma_1-\sigma_3}\begin{cases}
1 & (\sigma_2>1) \\
\log t_3 & (\sigma_2=1) \\
t_3^{1-\sigma_2} & (\sigma_2<1),  \\
\end{cases}
\end{equation}
when $s_1+s_3 \neq 0$. The term $D_{132}$ is meromorphically continued to the region $\sigma_1+\sigma_3>-\frac{1}{2}$ and $\sigma_3>-\frac{3}{2}$. and we have
\begin{equation}
\label{D_132}
D_{132} \ll t_3^{-\sigma_1-\sigma_3} \left( t_3^{-\sigma_3} +1 \right) \begin{cases}
1 & (\sigma_2>\frac{3}{2}) \\
\log t_3 & (\sigma_2=\frac{3}{2}) \\
t_3^{\frac{3}{2}-\sigma_2} & (\sigma_2<\frac{3}{2}) \\
\end{cases} \\
\end{equation}
in the region $\sigma_1+\sigma_3>-\frac{1}{2}, \sigma_3>-\frac{3}{2}$ and $s_1+s_3 \neq 0$.

Finally we estimate the term $D_{11}$. Using integration by parts, we have
\begin{align*}
D_{11} &= \frac{(at_3)^{1-s_1}}{s_1+s_3-1} \int_1^{at_3} \frac{1}{v^{s_2}(at_3+v)^{s_3}} dv \\
&+  \frac{s_3}{s_1+s_3-1} \int_1^{at_3} \frac{1}{v^{s_2-1}} \int_{at_3}^\infty \frac{1}{u^{s_1} (u+v)^{s_3+1}} dudv \\
&= D_{111}+D_{112},
\end{align*}
say.  Then we have
\begin{equation}
\label{D_111}
D_{111} \ll \frac{t_3^{1-\sigma_1}}{t_3} \int_1^{at_3} \frac{1}{v^{\sigma_2}(at_3+v)^{\sigma_3}} dv \ll \begin{cases}
t_3^{-\sigma_1-\sigma_3} & (\sigma_2>1) \\
t_3^{-\sigma_1-\sigma_3} & (\sigma_2=1) \\
t_3^{1-\sigma_1-\sigma_2-\sigma_3} & (\sigma_2<1) \\
\end{cases}
\end{equation}
when $s_1+s_3 \neq 1$.

Similar to $D_{13}$, we find that the term $D_{112}$ can be meromorphically continued to the region $\sigma_1+\sigma_3>-\frac{1}{2}, \sigma_3>-\frac{3}{2}$ and $s_1+s_3 \neq 0,1$, and in the same region we have
\begin{align}
\label{D_112}
D_{112} 
&\ll \left( t_3^{-\sigma_3} + 1 \right) \begin{cases}
t_3^{-\sigma_1-\sigma_3-1} & (\sigma_2>\frac{5}{2}) \\
t_3^{-\sigma_1-\sigma_3-1}\log t_3 & (\sigma_2=\frac{5}{2}) \\
t_3^{\frac{3}{2}-\sigma_1-\sigma_2-\sigma_3} & (\sigma_2<\frac{5}{2}). \\
\end{cases}
\end{align}

Therefore, by (\ref{D_14}), (\ref{D_121}), (\ref{D_122}), (\ref{D_131}), (\ref{D_132}), (\ref{D_111}), (\ref{D_112}), $D_1$ is analytic in the region $\sigma_1+\sigma_3>0, \sigma_3>-\frac{3}{2}$ and $s_1+s_3 \neq 1$, and in same region we have
\begin{equation}
\label{D_1}
D_1 \ll ( t_3^{-\sigma_3} +1 )\begin{cases}
t_3^{-\sigma_1-\sigma_3}  & (\sigma_2>\frac{3}{2}) \\
t_3^{-\sigma_1-\sigma_3} \log t_3 & (\sigma_2=\frac{3}{2}) \\
t_3^{\frac{3}{2}-\sigma_1-\sigma_2-\sigma_3} & (\sigma_2<\frac{3}{2}). \\
\end{cases}
\end{equation}

From (\ref{AV_2}), (\ref{D_2}), (\ref{D_1}), we have 
\begin{align}
\zeta_{AV,2} (s_1,s_2,s_3) &= \sum_{m \leq at_3} \sum_{n<m} \frac{1}{m^{s_1}n^{s_2}(m+n)^{s_3}} + D_1+D_2 \\
&+\sum_{n \leq at_3} \frac{E(s_1,s_3;at_3,n,M)}{n^{s_2}}+\sum_{at_3 < n \leq N} \frac{E(s_1,s_3;n,M)}{n^{s_2}} \nonumber \\
&-B_2-B_3-B_4-C_2-C_3-C_4 \nonumber
\end{align}
in the region $\sigma_1+\sigma_3>\frac{1}{2}, \sigma_1+\sigma_2+\sigma_3>\frac{3}{2}, \sigma_1 \geq 0, \sigma_3>0, t_1 \geq 0, t_3 \geq 2, s_1+s_3 \neq 1$ and $s_1+s_2+s_3 \neq 2$. 

Similar to the proof of Theorem \ref{thm:approximation1}, in the above region we also have 
\begin{align*}
&\sum_{n \leq at_3} \frac{E(s_1,s_3;at_3,n,M)}{n^{s_2}}+\sum_{at_3 < n \leq N} \frac{E(s_1,s_3;n,M)}{n^{s_2}} \nonumber \\
&-B_2-B_3-B_4-C_2-C_3-C_4 \\
& \ll \begin{cases}
t_3^{-\sigma_1-\sigma_3}  & (\sigma_2>1) \\
t_3^{-\sigma_1-\sigma_3} \log t_3 & (\sigma_2=1) \\
t_3^{1-\sigma_1-\sigma_2-\sigma_3} & (\sigma_2<1). \\
\end{cases}
\end{align*}
Since  $\sigma_3 > 0$, we find that $t_3^{-\sigma_3} \ll 1$. Then we obtain
\[
D_1+D_2 \ll \begin{cases}
t_3^{-\sigma_1-\sigma_3}  & (\sigma_2>\frac{3}{2}) \\
t_3^{-\sigma_1-\sigma_3} \log t_3 & (\sigma_2=\frac{3}{2}) \\
t_3^{\frac{3}{2}-\sigma_1-\sigma_2-\sigma_3} & (\sigma_2<\frac{3}{2}). \\
\end{cases}
\]
Thus we complete the proof of Theorem \ref{thm:approximation2}.  
\end{proof}

\section{Proofs of the main results}\label{sec3}

First, we give the following lemma to prove Theorem \ref{thm:first thm}. 
\begin{lemma}({\cite[Corollary 3]{MV}})
\label{lem:dirichlet poly}
For $a_1,a_2,\cdots,a_N \in \mathbb{C}$, we have
\[
\int_2^T \abs*{ \sum_{n \leq N} a_n n^{it} }^2 dt = T \sum_{n \leq N} \abs{a_n}^2 +O\left( \sum_{n \leq N} n \abs{a_n}^2 \right) \ \ \ (T \to \infty).
\]
and the above formula remains valid if $N=\infty$, provided that the series on the right-hand side of the above formula converge. 
\end{lemma}

%%%%Proof of first thm
\begin{proof}[Proof of Theorem \ref{thm:first thm}]
Let $s_j \in \mathbb{C}$ ($j=1,2,3$) with $\sigma_1+\sigma_3>1$ and $\sigma_1+\sigma_2+\sigma_3>2$. Then we have
\begin{align*}
\zeta_{AV,2} (s_1,s_2,s_3) &= \sum_{k=3}^\infty \left( \sum_{k/2 <m \leq k-1} \frac{1}{m^{s_1}(k-m)^{s_2}k^{\sigma_3}} \right) \frac{1}{k^{it_3}} \nonumber \\
&= \overline{\sum_{k=3}^\infty \overline{a_k} k^{it_3}} \nonumber \\
\end{align*}
where
\[
a_k = \sum_{k/2 <m \leq k-1} \frac{1}{m^{s_1}(k-m)^{s_2}k^{\sigma_3}} .
\]
Then we have
\[
\sum_{k=3}^\infty k \abs{a_k}^2 = \sum_{k=3}^\infty \abs*{ \sum_{k/2 <m \leq k-1} \frac{1}{m^{s_1}(k-m)^{s_2}} }^2 k^{1-2\sigma_3} = \zeta_{AV,2}^{[2]} (s_1,s_2,2\sigma_3-1). 
\]
By Theorem \ref{thm:series conv}, the series $\zeta_{AV,2}^{[2]} (s_1,s_2,2\sigma_3-1)$ is absolutely convergent in the region $\sigma_1+\sigma_3>1$, $\sigma_1+\sigma_2+\sigma_3>2$. Hence we can obtain $\sum_{k=2}^\infty k \abs{a_k}^2 = O(1)$. Also we have 
\[
\sum_{k=2}^\infty \abs{a_k}^2 = \zeta_{AV,2}^{[2]} (s_1,s_2,2\sigma_3). 
\]
Therefore, by Lemma \ref{lem:dirichlet poly}, we complete the proof of Theorem \ref{thm:first thm}. 
\end{proof}

Next, we prove Theorem \ref{thm:second thm} that is a mean theorem outside the region of absolute convergence. 

%%%Proof of second thm
\begin{proof}[Proof of Theorem \ref{thm:second thm}]
Assume that $\sigma_1\geq 0, \sigma_3>0, \frac{1}{2} < \sigma_1+\sigma_3\leq1, \sigma_1+\sigma_2+\sigma_3>2, t_3 \geq 2, t_1 \geq 0$ and $s_1+s_3 \neq 1$. Let $a=\max \{1, t_1 \}, x =y= at_3$ and $\kappa=\frac{4\pi}{3}$ in Theorem \ref{thm:approximation1}. Then we have
\begin{align}
\label{+*+}
&\zeta_{AV,2} (s_1,s_2,s_3) \\
&= \sum_{n \leq at_3} \sum_{n<m\leq at_3} \frac{1}{m^{s_1}n^{s_2}(m+n)^{s_3}} \nonumber \\
&+\frac{(at_3)^{1-s_1}}{s_1+s_3-1} \sum_{n \leq at_3} \frac{1}{n^{s_2}(at_3+n)^{s_3}} +\frac{s_3}{s_1+s_3-1} \sum_{n \leq at_3} \frac{1}{n^{s_2-1}} \int_{at_3}^\infty \frac{du}{u^{s_1}(u+n)^{s_3+1}} \nonumber \\
&+\frac{2^{-s_3}}{s_1+s_3-1} \sum_{n>at_3} \frac{1}{n^{s_1+s_2+s_3-1}} +\frac{s_3}{s_1+s_3-1} \sum_{n>at_3} \frac{1}{n^{s_2-1}} \int_n^\infty \frac{du}{u^{s_1}(u+n)^{s_3+1}} \nonumber \\
&+\begin{cases}
O \left( t_3^{-\sigma_1-\sigma_3} \right) & (\sigma_2>1) \\
O \left( t_3^{-\sigma_1-\sigma_3} \log t_3 \right) & (\sigma_2=1) \\
O \left( t_3^{1-\sigma_1-\sigma_2-\sigma_3} \right) & (\sigma_2<1). \\
\end{cases} \nonumber
\end{align}
As for the second term on the right hand side in (\ref{+*+}), we have 
\begin{align*}
\frac{(at_3)^{1-s_1}}{s_1+s_3-1} \sum_{n \leq at_3} \frac{1}{n^{s_2}(at_3+n)^{s_3}} &\ll \begin{cases}
t_3^{-\sigma_1-\sigma_3} & (\sigma_2>1) \\
t_3^{-\sigma_1-\sigma_3} \log t_3 & (\sigma_2=1) \\
t_3^{1-\sigma_1-\sigma_2-\sigma_3} & (\sigma_2<1). \\
\end{cases}
\end{align*}
Also, we find that
\begin{align*} 
\frac{2^{-s_3}}{s_1+s_3-1} \sum_{n>at_3} \frac{1}{n^{s_1+s_2+s_3-1}} &\ll t_3^{1-\sigma_1-\sigma_2-\sigma_3}.
\end{align*}
The third term and the fifth term are estimated as
\begin{align*}
\frac{s_3}{s_1+s_3-1} \sum_{n \leq at_3}  \frac{1}{n^{s_2-1}} \int_{at_3}^\infty \frac{du}{u^{s_1}(u+n)^{s_3+1}} &\ll \begin{cases}
t_3^{-\sigma_1-\sigma_3} & (\sigma_2>2) \\
t_3^{-\sigma_1-\sigma_3} \log t_3 & (\sigma_2=2) \\
t_3^{2-\sigma_1-\sigma_2-\sigma_3} & (\sigma_2<2), \\
\end{cases}
\end{align*}
and 
\begin{align*}
\frac{s_3}{s_1+s_3-1} \sum_{n>at_3} \frac{1}{n^{s_2-1}} \int_n^\infty \frac{du}{u^{s_1}(u+n)^{s_3+1}} &\ll t_3^{2-\sigma_1-\sigma_2-\sigma_3}.
\end{align*}

In the case $\sigma_2<2$, we can develop more elaborate arguments for the third and the fifth terms by using the method of Okamoto and Onozuka \cite{OO}. 

By the Mellin-Barnes integral formula, the third term is transformed as
\begin{align*}
& \frac{s_3}{s_1+s_3-1} \sum_{n \leq at_3} \frac{1}{n^{s_2-1}} \int_{at_3}^\infty \frac{du}{u^{s_1}(u+n)^{s_3+1}} \\
&= \frac{s_3}{2\pi i (s_1+s_3-1) \Gamma(s_3+1)}\sum_{n \leq at_3} \int_{at_3}^\infty \int_{(c)} \frac{\Gamma(s_3+1+z)\Gamma(-z)}{u^{s_1+s_3+1+z} n^{s_2-1-z}} dz du, \\
\end{align*}
where $-\sigma_1-\sigma_3<c<0$. Since $\sigma_1+\sigma_3+c>0$, we can change the order of the integration and summation. Then we have
\begin{align*}
& \frac{s_3}{2\pi i (s_1+s_3-1) \Gamma(s_3+1)} \\
& \times \int_{(c)} \frac{\Gamma(s_3+1+z)\Gamma(-z)}{(s_1+s_3+z) (at_3)^{s_1+s_3+z}} \left( \sum_{n \leq at_3} \frac{1}{n^{s_2-1-z}} \right) dz .
\end{align*}
Shifting the path $(c)$ to $\left( \frac{1}{2} \right)$, the above formula is equal to
\begin{align*}
&= \frac{s_3}{(s_1+s_3-1)(s_1+s_3)(at_3)^{s_1+s_3}} \sum_{n \leq at_3} \frac{1}{n^{s_2-1}} \\
&+ \frac{s_3}{2\pi i (s_1+s_3-1) \Gamma(s_3+1)} \int_{\left( \frac{1}{2} \right)} \frac{\Gamma(s_3+1+z)\Gamma(-z)}{(s_1+s_3+z) (at_3)^{s_1+s_3+z}} \left( \sum_{n \leq at_3} \frac{1}{n^{s_2-1-z}} \right) dz.
\end{align*}
Since $\sigma_2<2$, we see that
\begin{align*}
\frac{s_3}{(s_1+s_3-1)(s_1+s_3)} (at_3)^{-s_1-s_3} \sum_{n \leq at_3} \frac{1}{n^{s_2-1}} &\ll t_3^{1-\sigma_1-\sigma_2-\sigma_3},
\end{align*}
and since $\sigma_3>0$, we have
\begin{align*}
&\frac{s_3}{2\pi i (s_1+s_3-1) \Gamma(s_3+1)} \int_{\left( \frac{1}{2} \right)} \frac{\Gamma(s_3+1+z)\Gamma(-z)}{(s_1+s_3+z) (at_3)^{s_1+s_3+z}} \sum_{n \leq at_3} \frac{1}{n^{s_2-1-z}} dz \\
&\quad \ll t_3^{-\sigma_1-\sigma_3-\frac{1}{2}} \left( \sum_{n \leq at_3} \frac{1}{n^{\sigma_2-\frac{3}{2}}} \right) \frac{1}{\abs{\Gamma(s_3+1)}} \int_{\left( \frac{1}{2} \right)} \abs*{ \frac{\Gamma(s_3+1+z)\Gamma(-z)}{s_1+s_3+z}} dz \\
&\quad \ll t_3^{\frac{3}{2}-\sigma_1-\sigma_2-\sigma_3}.
\end{align*}

Similarly we estimate the fifth term in the case $\sigma_2<2$. Let $-\sigma_1-\sigma_3<c<\sigma_2-2$, then we have
\begin{align*}
& \frac{s_3}{2\pi i (s_1+s_3-1) \Gamma(s_3+1)}\sum_{n>at_3} \int_n^\infty \int_{(c)} \frac{\Gamma(s_3+1+z)\Gamma(-z)}{u^{s_1+s_3+1+z} n^{s_2-1-z}} dz du. \\
&= \frac{s_3}{2\pi i (s_1+s_3-1) \Gamma(s_3+1)} \sum_{n>at_3} \frac{1}{n^{s_1+s_2+s_3-1}} \int_{(c)} \frac{\Gamma(s_3+1+z)\Gamma(-z)}{s_1+s_3+z} dz .
\end{align*}
Shifting the path $(c)$ to $\left( \frac{1}{2} \right)$, the above formula is equal to
\begin{align*}
&\frac{s_3}{(s_1+s_3-1)(s_1+s_3)} \sum_{n>at_3} \frac{1}{n^{s_1+s_2+s_3-1}} \\
&+ \frac{s_3}{2\pi i (s_1+s_3-1) \Gamma(s_3+1)} \sum_{n>at_3} \frac{1}{n^{s_1+s_2+s_3-1}}\int_{\left( \frac{1}{2} \right)} \frac{\Gamma(s_3+1+z)\Gamma(-z)}{s_1+s_3+z} dz .
\end{align*}
Then we have 
\begin{align*}
\frac{s_3}{(s_1+s_3-1)(s_1+s_3)} \sum_{n>at_3} \frac{1}{n^{s_1+s_2+s_3-1}} &\ll t_3^{1-\sigma_1-\sigma_2-\sigma_3},
\end{align*}
and 
\begin{align*}
&\frac{s_3}{2\pi i (s_1+s_3-1) \Gamma(s_3+1)} \sum_{n>at_3} \frac{1}{n^{s_1+s_2+s_3-1}}\int_{\left( \frac{1}{2} \right)} \frac{\Gamma(s_3+1+z)\Gamma(-z)}{s_1+s_3+z} dz \\ 
&\ll t_3^{\frac{3}{2}-\sigma_1-\sigma_2-\sigma_3}
\end{align*}
since $\sigma_3>0$.

Summarizing the above results, we find that
\begin{align}
\label{approximation1-1}
\zeta_{AV,2} (s_1,s_2,s_3) &= \sum_{m \leq at_3} \sum_{n<m} \frac{1}{m^{s_1}n^{s_2}(m+n)^{s_3}}+ \begin{cases}
O (t_3^{-\sigma_1-\sigma_3} ) & (\sigma_2>2) \\
O (t_3^{-\sigma_1-\sigma_3} \log t_3 ) & (\sigma_2=2) \\
O (t_3^{\frac{3}{2}-\sigma_1-\sigma_2-\sigma_3} ) & (\sigma_2<2). \\
\end{cases}
\end{align}

Now we denote the main term of (\ref{approximation1-1}) by $\Sigma_1(s_1,s_2,s_3)$ and the error term by $E(s_1,s_2,s_3)$, respectively. Then, for $M(m_1,m_2)=\max \{ m_1/a,m_2/a,2\}$,
\begin{align}
\label{w_1Tw_2+w_3}
& \int_2^T \abs{\Sigma_1(s_1,s_2,s_3)}^2 dt_3 \\
&= \int_2^T \sum_{m_1,m_2 \leq at_3} \sum_{\substack{n_1<m_1 \\ n_2<m_2}} \frac{1}{m_1^{s_1}m_2^{\overline{s_2}}n_1^{s_2}n^{\overline{s_2}}(m_1+n_1)^{\sigma_3+it_3}(m_2+n_2)^{\sigma_3-it_3}} dt_3 \nonumber \\
&= \sum_{m_1,m_2 \leq aT} \sum_{\substack{n_1<m_1 \\ n_2<m_2}} \frac{1}{m_1^{s_1}m_2^{\overline{s_2}}n_1^{s_2}n_2^{\overline{s_2}}(m_1+n_1)^{\sigma_3}(m_2+n_2)^{\sigma_3}} \nonumber \\ 
&\times \int_{M(m_1,m_2)}^T \left( \frac{m_2+n_2}{m_1+n_1} \right)^{it_3} dt_3 \nonumber \\
&= \sum_{m_1,m_2 \leq aT} \sum_{\substack{n_1<m_1,n_2<m_2 \\ m_1+n_1=m_2+n_2}} \frac{1}{m_1^{s_1}m_2^{\overline{s_1}}n_1^{s_2}n_2^{\overline{s_2}}(m_1+n_1)^{2\sigma_3}} \left( T -M(m_1,m_2) \right) \nonumber \\
&\quad+ \sum_{m_1,m_2 \leq aT} \sum_{\substack{n_1<m_1,n_2<m_2 \\ m_1+n_1=m_2+n_2}} \frac{1}{m_1^{s_1}m_2^{\overline{s_2}}n_1^{s_2}n_2^{\overline{s_2}}(m_1+n_1)^{\sigma_3}(m_2+n_2)^{\sigma_3}} \nonumber \\
&\qquad \times \frac{e^{iT \log \left( \frac{m_2+n_2}{m_1+n_1} \right)}-e^{iM(m_1,m_2) \log \left( \frac{m_2+n_2}{m_1+n_1} \right)}}{i \log \left( \frac{m_2+n_2}{m_1+n_1}\right)} \nonumber \\
&=W_1 T-W_2+W_3, \nonumber 
\end{align}
say. 

First we calculate $W_1$.
\begin{align*}
W_1 &= \sum_{m_1,m_2 \leq aT} \sum_{\substack{n_1<m_1 \\ n_2<m_2 \\ m_1+n_1=m_2+n_2}} \frac{1}{m_1^{s_1}m_2^{\overline{s_1}}n_1^{s_2}n_2^{\overline{s_2}}(m_1+n_1)^{2\sigma_3}} \\
&= \left( \sum_{m_1,m_2 \geq 1} \sum_{\substack{n_1<m_1 \\ n_2<m_2 \\ m_1+n_1=m_2+n_2}}-\sum_{\substack{m_1>aT \\ m_2 \leq aT}} \sum_{\substack{n_1<m_1 \\ n_2<m_2 \\ m_1+n_1=m_2+n_2}} \right. \\
& \left. -\sum_{\substack{m_1\leq aT \\ m_2 > aT}} \sum_{\substack{n_1<m_1 \\ n_2<m_2 \\ m_1+n_1=m_2+n_2}}-\sum_{m_1,m_2 > aT} \sum_{\substack{n_1<m_1 \\ n_2<m_2 \\ m_1+n_1=m_2+n_2}} \right) \\
&\times \frac{1}{m_1^{s_1}m_2^{\overline{s_1}}n_1^{s_2}n_2^{\overline{s_2}}(m_1+n_1)^{2\sigma_3}} \\
&= \sum_{m_1,m_2 \geq 1} \sum_{\substack{n_1<m_1 \\ n_2<m_2 \\ m_1+n_1=m_2+n_2}} \frac{1}{m_1^{s_1}m_2^{\overline{s_1}}n_1^{s_2}n_2^{\overline{s_2}}(m_1+n_1)^{2\sigma_3}} -U_1-U_2-U_3, \\
\end{align*}
say. Then the first term is identical with $\zeta_{AV,2}^{[2]} (s_1,s_2,2\sigma_3)$. Note that $U_1+U_2+U_3=(U_1+U_3)+(U_2+U_3)-U_3$.
We see that $n_2<\frac{m_1+n_1}{2}$ since $n_2<m_2=m_1+n_1-n_2$, and also see that $\sigma_2>1$ since $1 \geq \sigma_1+\sigma_3>\frac{1}{2}$ and $\sigma_1+\sigma_2+\sigma_3>2$, and hence we have  
\begin{align*}
U_1+U_3 &\ll \sum_{\substack{m_1>aT \\ m_2 \geq 1}} \sum_{\substack{n_1<m_1 \\ n_2<m_2 \\ m_1+n_1=m_2+n_2}} \frac{1}{m_1^{\sigma_1}m_2^{\sigma_1}n_1^{\sigma_2}n_2^{\sigma_2}(m_1+n_1)^{2\sigma_3}} \\
&\ll \sum_{\substack{m_1>aT \\ n_1 < m_1}} \sum_{n_2 <\frac{m_1+n_1}{2}} \frac{1}{m_1^{\sigma_1}(m_1+n_1-n_2)^{\sigma_1}n_1^{\sigma_2}n_2^{\sigma_2}(m_1+n_1)^{2\sigma_3}} \\
&\ll \sum_{\substack{m_1>aT \\ n_1 < m_1}} \frac{1}{m_1^{\sigma_1}n_1^{\sigma_2}(m_1+n_1)^{\sigma_1+2\sigma_3}} \sum_{n_2 <\frac{m_1+n_1}{2}} \frac{1}{n_2^{\sigma_2}} \\
&\ll \sum_{\substack{m_1>aT \\ n_1 < m_1}} \frac{1}{m_1^{\sigma_1}n_1^{\sigma_2}(m_1+n_1)^{\sigma_1+2\sigma_3}} \\
&\ll \sum_{m_1>aT} \frac{1}{m_1^{2\sigma_1+2\sigma_3}} \sum_{n_1 < m_1} \frac{1}{n_1^{\sigma_2}} \\
&\ll \sum_{m_1>aT} \frac{1}{m_1^{2\sigma_1+2\sigma_3}} \\
&\ll T^{1-2\sigma_1-2\sigma_3}.
\end{align*}

Similarly we have $U_2+U_3 , U_3= O \left( T^{1-2\sigma_1-2\sigma_3} \right)$. Then we obtain
\begin{equation}
\label{w_1 T}
W_1 T = \zeta_{AV,2}^{[2]} (s_1,s_2,2\sigma_3)T+ O \left( T^{2-2\sigma_1-2\sigma_3} \right).
\end{equation}

Next, we estimate $W_2$. Since $M(m_1,m_2) \leq m_1+n_1 (=m_2+n_2)$ and $n_2 <\frac{m_1+n_1}{2}, \sigma_2>1$, we have
\begin{align}
\label{w_2}
W_2 &\ll \sum_{m_1,m_2 \leq aT} \sum_{\substack{n_1<m_1 \\ n_2<m_2 \\ m_1+n_1 =m_2+n_2}} \frac{1}{m_1^{\sigma_1}m_2^{\sigma_1}n_1^{\sigma_2}n_2^{\sigma_2}(m_1+n_1)^{2\sigma_3-1}} \nonumber \\
&\ll \sum_{\substack{m_1 \leq aT \\ n_1<m_1}} \sum_{n_2<\frac{m_1+n_1}{2}} \frac{1}{m_1^{\sigma_1}(m_1+n_1-n_2)^{\sigma_1}n_1^{\sigma_2}n_2^{\sigma_2}(m_1+n_1)^{2\sigma_3-1}} \nonumber \\
&\ll \sum_{\substack{m_1 \leq aT \\ n_1 < m_1}} \frac{1}{m_1^{\sigma_1}n_1^{\sigma_2}(m_1+n_1)^{\sigma_1+2\sigma_3-1}} \sum_{n_2 <\frac{m_1+n_1}{2}} \frac{1}{n_2^{\sigma_2}} \nonumber \\
&\ll \sum_{\substack{m_1 \leq aT \\ n_1 < m_1}} \frac{1}{m_1^{\sigma_1}n_1^{\sigma_2}(m_1+n_1)^{\sigma_1+2\sigma_3-1}} \nonumber \\
&\ll \sum_{m_1 \leq aT} \frac{1}{m_1^{2\sigma_1+2\sigma_3-1}} \sum_{n_1 < m_1} \frac{1}{n_1^{\sigma_2}} \nonumber \\
&\ll \sum_{m_1 \leq aT} \frac{1}{m_1^{2\sigma_1+2\sigma_3-1}} \nonumber \\
&\ll T^{2-2\sigma_1-2\sigma_3}.
\end{align}

Finally we estimate $W_3$.
\begin{align*}
W_3 &= \sum_{m_1,m_2 \leq aT} \sum_{\substack{n_1<m_1 \\ n_2<m_2 \\ m_1+n_1 \neq m_2+n_2}} \frac{1}{m_1^{s_1}m_2^{\overline{s_1}}n_1^{s_2}n_2^{\overline{s_2}}(m_1+n_1)^{\sigma_3}(m_2+n_2)^{\sigma_3}} \\
&\quad \times \frac{e^{iT \log \left( \frac{m_2+n_2}{m_1+n_1} \right)}-e^{iM(m_1,m_2) \log \left( \frac{m_2+n_2}{m_1+n_1} \right)}}{i \log \left( \frac{m_2+n_2}{m_1+n_1}\right)} \\
&\ll \sum_{m_1,m_2 \leq aT} \sum_{\substack{n_1<m_1 \\ n_2<m_2 \\ m_1+n_1 < m_2+n_2 \leq 2(m_1+n_1)}} \frac{1}{m_1^{\sigma_1}m_2^{\sigma_1}n_1^{\sigma_2}n_2^{\sigma_2}} \\
& \times \frac{1}{(m_1+n_1)^{\sigma_3}(m_2+n_2)^{\sigma_3}} \frac{1}{\log \left( \frac{m_2+n_2}{m_1+n_1}\right)} \\
&+ \sum_{m_1,m_2 \leq aT} \sum_{\substack{n_1<m_1 \\ n_2<m_2 \\ 2(m_1+n_1)< m_2+n_2}} \frac{1}{m_1^{\sigma_1}m_2^{\sigma_1}n_1^{\sigma_2}n_2^{\sigma_2}} \\
&\times \frac{1}{(m_1+n_1)^{\sigma_3}(m_2+n_2)^{\sigma_3}} \frac{1}{\log \left( \frac{m_2+n_2}{m_1+n_1}\right)}.
\end{align*}
We denote the first and the second terms by $W_{32}$ and $W_{31}$, respectively. As for $W_{32}$, we have
\begin{align}
\label{w_32}
W_{32} &\ll \sum_{m_1,m_2 \leq aT} \sum_{\substack{n_1<m_1 \\ n_2<m_2}} \frac{1}{m_1^{\sigma_1}m_2^{\sigma_1}n_1^{\sigma_2}n_2^{\sigma_2}(m_1+n_1)^{\sigma_3}(m_2+n_2)^{\sigma_3}} \\
&\ll \left( \sum_{m \leq aT} \frac{1}{m^{\sigma_1}} \sum_{n<m} \frac{1}{n^{\sigma_2}(m+n)^{\sigma_3}} \right)^2 \nonumber \\
&\ll \left( \sum_{m \leq aT} \frac{1}{m^{\sigma_1+\sigma_3}} \sum_{n<m} \frac{1}{n^{\sigma_2}} \right)^2 \nonumber \\
&\ll \left( \begin{cases}
T^{1-\sigma_1-\sigma_3} & (\frac{1}{2}<\sigma_1+\sigma_3<1) \\
\log T & (\sigma_1+\sigma_3=1) \\
\end{cases}
\right)^2 \nonumber \\
&\ll \begin{cases}
T^{2-2\sigma_1-2\sigma_3} & (\frac{1}{2}<\sigma_1+\sigma_3<1) \\
(\log T)^2 & (\sigma_1+\sigma_3=1). \\
\end{cases} \nonumber 
\end{align}

As for $W_{31}$, setting $r=m_2+n_2-m_1-n_1$, we have $m_2=r+m_1+n_1-n_2$, and by $n_2<\frac{r+m_1+n_1}{2}$ and $\log (1+x) \asymp x$ ($\abs{x}<1$), we have 
\begin{align}
\label{w_31}
W_{31} &\ll \sum_{\substack{m_1 \leq aT \\ n_1<m_1}} \frac{1}{m_1^{\sigma_1}n_1^{\sigma_2}(m_1+n_1)^{\sigma_3}} \sum_{r \leq m_1+n_1} \sum_{n_2 <\frac{r+m_1+n_1}{2}} \\
& \times \frac{1}{n_2^{\sigma_2}(r+m_1+n_1-n_2)^{\sigma_1}(m_1+n_1+r)^{\sigma_3}} \frac{1}{\log \left( 1+\frac{r}{m_1+n_1}\right)} \nonumber \\
&\ll \sum_{\substack{m_1 \leq aT \\ n_1<m_1}} \frac{1}{m_1^{\sigma_1}n_1^{\sigma_2}(m_1+n_1)^{\sigma_3}} \sum_{r \leq m_1+n_1} \sum_{n_2 <\frac{r+m_1+n_1}{2}} \nonumber \\
&\times \frac{1}{n_2^{\sigma_2}(r+m_1+n_1-n_2)^{\sigma_1}(m_1+n_1+r)^{\sigma_3}} \frac{m_1+n_1}{r} \nonumber \\
&\ll \sum_{\substack{m_1 \leq aT \\ n_1<m_1}} \frac{1}{m_1^{\sigma_1}n_1^{\sigma_2}(m_1+n_1)^{\sigma_3-1}} \nonumber \\ 
&\quad +\sum_{r \leq m_1+n_1} \frac{1}{r(m_1+n_1+r)^{\sigma_1+\sigma_3}} \sum_{n_2 <\frac{r+m_1+n_1}{2}} \frac{1}{n_2^{\sigma_2}} \nonumber \\
&\ll \sum_{\substack{m_1 \leq aT \\ n_1<m_1}} \frac{1}{m_1^{\sigma_1}n_1^{\sigma_2}(m_1+n_1)^{\sigma_1+2\sigma_3-1}} \sum_{r \leq m_1+n_1} \frac{1}{r} \nonumber \\
&\ll \sum_{\substack{m_1 \leq aT \\ n_1<m_1}} \frac{\log (m_1+n_1)}{m_1^{\sigma_1}n_1^{\sigma_2}(m_1+n_1)^{\sigma_1+2\sigma_3-1}} \nonumber \\
&\ll \sum_{m_1 \leq aT} \frac{1}{m_1^{\sigma_1}} \sum_{n_1<m_1} \frac{\log (m_1+n_1)}{n_1^{\sigma_2}(m_1+n_1)^{\sigma_1+2\sigma_3-1}} \nonumber \\
&\ll \sum_{m_1 \leq aT} \frac{\log m_1}{m_1^{2\sigma_1+2\sigma_3-1}} \sum_{n_1<m_1} \frac{1}{n_1^{\sigma_2}} \nonumber \\
&\ll \log T \sum_{m_1 \leq aT} \frac{1}{m_1^{2\sigma_1+2\sigma_3-1}} \nonumber \\
&\ll \begin{cases}
T^{2-2\sigma_1-2\sigma_3} \log T & (\frac{1}{2} < \sigma_1+\sigma_3 <1) \\
(\log T)^2 & (\sigma_1+\sigma_3=1). \\
\end{cases} \nonumber
\end{align}

Therefore by (\ref{w_1 T}), (\ref{w_2}), (\ref{w_32}) and (\ref{w_31}), we obtain
\begin{align*}
\int_2^T \abs{\Sigma_1 (s_1,s_2,s_3)}^2 dt_3 &=\zeta_{AV,2}^{[2]} (s_1,s_2,2\sigma_3)T \\
&+ \begin{cases}
O \left( T^{2-2\sigma_1-2\sigma_3} \log T \right) & (\frac{1}{2} < \sigma_1+\sigma_3 <1) \\
O (\log T)^2 & (\sigma_1+\sigma_3=1). \\
\end{cases}
\end{align*}

Then
\begin{align*}
&\int_2^T \abs{\zeta_{AV,2} (s_1,s_2,s_3) }^2 dt_3 \\
&= \int_2^T \abs{\Sigma_1 (s_1,s_2,s_3) + E(s_1,s_2,s_3)}^2 dt_3 \\
&= \int_2^T \abs{\Sigma_1 (s_1,s_2,s_3)}^2 dt_3 + O \left( \int_2^T \abs{\Sigma_1(s_1,s_2,s_3)E(s_1,s_2,s_3)} dt_3 \right) \\
&\quad+O \left( \int_2^T \abs{E(s_1,s_2,s_3)}^2 dt_3 \right).  
\end{align*}
Since $\frac{1}{2}<\sigma_1+\sigma_3\leq 1$ and $\sigma_1+\sigma_2+\sigma_3>2$, we see that
\begin{align*}
\int_2^T \abs{E(s_1,s_2,s_3)}^2 dt_3 &= \begin{cases}
O \left( \int_2^T t_3^{-2\sigma_1-2\sigma_3} dt_3 \right) & (\sigma_2 >2) \\
O \left( \int_2^T t_3^{-2\sigma_1-2\sigma_3} (\log T)^2 dt_3 \right) & (\sigma_2=2) \\
O \left( \int_2^T t_3^{3-2\sigma_1-2\sigma_2-2\sigma_3} dt_3 \right) & (\sigma_2<2) \\
\end{cases} \\
&=O(1), 
\end{align*}
and 
\begin{align*}
&\int_2^T \abs{\Sigma_1(s_1,s_2,s_3)E(s_1,s_2,s_3)} dt_3 \\
&\ll \left( \int_2^T \abs{\Sigma_1(s_1,s_2,s_3)}^2 dt_3 \right)^\frac{1}{2} \left( \int_2^T \abs{E(s_1,s_2,s_3)}^2 dt_3 \right)^\frac{1}{2} \\
&\ll T^\frac{1}{2} \cdot 1^\frac{1}{2} \ll T^\frac{1}{2}. \\
\end{align*}

Thus we have
\begin{align*}
\int_2^T \abs{\zeta_{AV,2} (s_1,s_2,s_3) }^2 dt_3 &= \zeta_{AV,2}^{[2]} (s_1,s_2,2\sigma_3)T \\
&+\begin{cases}
O ( T^{2-2\sigma_1-2\sigma_3} \log T ) & (\frac{1}{2} < \sigma_1+\sigma_3 \leq \frac{3}{4}) \\
O ( T^\frac{1}{2} ) & (\frac{3}{4} < \sigma_1+\sigma_3 \leq 1), \\
\end{cases}
\end{align*}
so we complete the proof of Theorem \ref{thm:second thm}.  
\end{proof}

Next, we give a mean square formula for $\sigma_1+\sigma_3>\frac{1}{2}$ and $2 \geq \sigma_1+\sigma_2+\sigma_3 > \frac{3}{2}$ by using Theorem \ref{thm:approximation2}.
%%%Proof of third thm
\begin{proof}[Proof of Theorem \ref{thm:third thm}]
For $s_j =\sigma_j+it_j \in \mathbb{C} (j = 1,2,3)$ with $\sigma_1 \geq 0, \sigma_3>0, \sigma_1+\sigma_3>\frac{1}{2}, 2 \geq \sigma_1+\sigma_2+\sigma_3 > \frac{3}{2}$. We also assume that $s_1+s_3 \neq 1$ and $s_1+s_2+s_3 \neq 2$. From the above assumptions, we find that $\sigma_2 <\frac{3}{2}$. Therefore by Theorem \ref{thm:approximation2} we have 
\begin{equation}
\label{av-3}
\zeta_{AV,2} (s_1,s_2,s_3) = \sum_{m \leq at_3} \sum_{n<m} \frac{1}{m^{s_1}n^{s_2}(m+n)^{s_3}} +
O ( t_3^{\frac{3}{2}-\sigma_1-\sigma_2-\sigma_3} ).
\end{equation}

Similar to the proof of Theorem \ref{thm:second thm}, denote the main term of (\ref{av-3}) by $\Sigma_1(s_1,s_2,s_3)$ and the error term by $E(s_1,s_2,s_3)$, respectively. Then in the same way as (\ref{w_1Tw_2+w_3}), we have
\[
\int_2^T \abs{\Sigma_1(s_1,s_2,s_3)}^2 dt_3 =W_1 T-W_2+W_3,
\]
\begin{align*}
W_1 
&= \zeta_{AV,2}^{[2]} (s_1,s_2,2\sigma_3)-(U_1+U_3)-(U_2+U_3)+U_3
\end{align*}
and
\begin{align*}
U_1+U_3 &\ll \sum_{\substack{m_1>aT \\ n_1 < m_1}} \frac{1}{m_1^{\sigma_1}n_1^{\sigma_2}(m_1+n_1)^{\sigma_1+2\sigma_3}} \sum_{n_2 <\frac{m_1+n_1}{2}} \frac{1}{n_2^{\sigma_2}}. 
\end{align*}
In the case of Theorem \ref{thm:second thm}, we considered only one case since $\sigma_2>1$. However now we have to consider the following three cases since $\sigma_2 <\frac{3}{2}$. We estimate each cases separately, and obtain
\begin{align*}
& \sum_{\substack{m_1>aT \\ n_1 < m_1}} \frac{1}{m_1^{\sigma_1}n_1^{\sigma_2}(m_1+n_1)^{\sigma_1+2\sigma_3}} \sum_{n_2 <\frac{m_1+n_1}{2}} \frac{1}{n_2^{\sigma_2}} \\
& \ll \sum_{\substack{m_1>aT \\ n_1 < m_1}} \begin{cases}
\frac{1}{m_1^{\sigma_1}n_1^{\sigma_2}(m_1+n_1)^{\sigma_1+2\sigma_3}} & (\frac{3}{2} >\sigma_2 >1) \\
\frac{\log (m_1+n_1)}{m_1^{\sigma_1}n_1^{\sigma_2}(m_1+n_1)^{\sigma_1+2\sigma_3}} & (\sigma_2 =1) \\
\frac{1}{m_1^{\sigma_1}n_1^{\sigma_2}(m_1+n_1)^{\sigma_1+\sigma_2+2\sigma_3-1}} & (\sigma_2 <1) \\
\end{cases} \\
& \ll \begin{cases}
T^{1-2\sigma_1-2\sigma_3} & (\frac{3}{2} >\sigma_2 >1) \\
T^{1-2\sigma_1-2\sigma_3} (\log T)^2 & (\sigma_2 =1) \\
T^{3-2\sigma_1-2\sigma_2-2\sigma_3} & (\sigma_2 <1). \\
\end{cases} 
\end{align*}
Similarly, we have
\[
U_2+U_3 , U_3 = \begin{cases}
O\left(T^{1-2\sigma_1-2\sigma_3} \right) & (\frac{3}{2} >\sigma_2 >1) \\
O\left(T^{1-2\sigma_1-2\sigma_3} (\log T)^2 \right) & (\sigma_2 =1) \\
O\left(T^{3-2\sigma_1-2\sigma_2-2\sigma_3} \right) & (\sigma_2 <1). \\
\end{cases} 
\]

Therefore we obtain
\begin{equation}
\label{W_1 T}
W_1 T = \zeta_{AV,2}^{[2]} (s_1,s_2,2\sigma_3)T+ \begin{cases}
O (T^{2-2\sigma_1-2\sigma_3} ) & (\frac{3}{2} >\sigma_2 >1) \\
O (T^{2-2\sigma_1-2\sigma_3} (\log T)^2 ) & (\sigma_2 =1) \\
O (T^{4-2\sigma_1-2\sigma_2-2\sigma_3} ) & (\sigma_2 <1). \\
\end{cases} 
\end{equation}

Next, by Theorem \ref{thm:second thm}, we have
\begin{align*}
W_2 &\ll \sum_{\substack{m_1 \leq aT \\ n_1 < m_1}} \frac{1}{m_1^{\sigma_1}n_1^{\sigma_2}(m_1+n_1)^{\sigma_1+2\sigma_3-1}} \sum_{n_2 <\frac{m_1+n_1}{2}} \frac{1}{n_2^{\sigma_2}}.
\end{align*}
Considering the same three cases as above, we have
\begin{align}
\label{W_2}
W_2 &\ll \sum_{m_1 \leq aT} \frac{1}{m_1^{\sigma_1}} \sum_{n_1<m_1} \begin{cases}
 \frac{1}{n_1^{\sigma_2} (m_1+n_1)^{\sigma_1+2\sigma_3-1}} & (\frac{3}{2} >\sigma_2 >1) \\
 \frac{\log (m_1+n_1)}{n_1^{\sigma_2} (m_1+n_1)^{\sigma_1+2\sigma_3-1}} & (\sigma_2=1) \\
 \frac{1}{n_1^{\sigma_2} (m_1+n_1)^{\sigma_1+\sigma_2+2\sigma_3-2}} & (\sigma_2<1) \\
\end{cases} \\
&\ll \sum_{m_1 \leq aT}\begin{cases}
 \frac{1}{m_1^{2\sigma_1+2\sigma_3-1}} & (\frac{3}{2} >\sigma_2 >1) \\
 \frac{(\log m_1)^2}{m_1^{2\sigma_1+2\sigma_3-1}} & (\sigma_2=1) \\
 \frac{1}{m_1^{2\sigma_1+2\sigma_2+2\sigma_3-3}} & (\sigma_2<1) \\
\end{cases} \nonumber \\
& \ll \begin{cases}
T^{2-2\sigma_1-2\sigma_3} & (\frac{3}{2} >\sigma_2 >1, \frac{1}{2}<\sigma_1+\sigma_3<1) \\
T^{2-2\sigma_1-2\sigma_3} (\log T)^2 & (\sigma_2=1, \frac{1}{2}<\sigma_1+\sigma_3<1) \\
(\log T)^3 & (\sigma_2=1, \sigma_1+\sigma_3=1) \\
\log T &  (\sigma_2<1, \sigma_1+\sigma_2+\sigma_3 = 2) \\
T^{4-2\sigma_1-2\sigma_2-2\sigma_3} & (\sigma_2<1, \frac{3}{2}<\sigma_1+\sigma_2+\sigma_3 < 2).
\end{cases} \nonumber 
\end{align}

Lastly we estimate $W_3$. We have
\begin{align*}
W_3 &\ll W_{31}+W_{32}
\end{align*}
where
\begin{align*}
W_{31} &\ll \sum_{\substack{m_1 \leq aT \\ n_1<m_1}} \frac{1}{m_1^{\sigma_1}n_1^{\sigma_2}(m_1+n_1)^{\sigma_3-1}} \\
&\times \sum_{r \leq m_1+n_1} \frac{1}{r(m_1+n_1+r)^{\sigma_1+\sigma_3}} \sum_{n_2 <\frac{r+m_1+n_1}{2}} \frac{1}{n_2^{\sigma_2}}, \\
\intertext{and}
W_{32} &\ll \left( \sum_{m \leq aT} \frac{1}{m^{\sigma_1+\sigma_3}} \sum_{n<m} \frac{1}{n^{\sigma_2}} \right)^2.
\end{align*}

First we estimate $W_{32}$. We see that
\begin{align}
\label{W_32}
W_{32} &\ll \left( \sum_{m \leq aT} \frac{1}{m^{\sigma_1+\sigma_3}} \sum_{n<m} \frac{1}{n^{\sigma_2}} \right)^2 \\
&\ll \left( \sum_{m \leq aT} \begin{cases}
\frac{1}{m^{\sigma_1+\sigma_3}} & (\frac{3}{2}>\sigma_2>1) \\
\frac{\log m}{m^{\sigma_1+\sigma_3}} & (\sigma_2=1) \\
\frac{1}{m^{\sigma_1+\sigma_2+\sigma_3-1}} & (\sigma_2<1) \\
\end{cases} 
\right)^2 \nonumber \\
&\ll \left( \begin{cases}
T^{1-\sigma_1-\sigma_3} & (\frac{3}{2}>\sigma_2>1, \frac{1}{2}<\sigma_1+\sigma_3<1) \\
T^{1-\sigma_1-\sigma_3} \log T & (\sigma_2=1, \frac{1}{2}<\sigma_1+\sigma_3<1) \\
(\log T)^2 & (\sigma_2=1, \sigma_1+\sigma_3=1) \\
\log T &  (\sigma_2<1, \sigma_1+\sigma_2+\sigma_3 = 2) \\
T^{2-\sigma_1-\sigma_2-\sigma_3} & (\sigma_2<1, \frac{3}{2}<\sigma_1+\sigma_2+\sigma_3 < 2)
\end{cases} 
\right)^2 \nonumber \\
&\ll \begin{cases}
T^{2-2\sigma_1-2\sigma_3} & (\frac{3}{2}>\sigma_2>1, \frac{1}{2}<\sigma_1+\sigma_3<1) \\
T^{2-2\sigma_1-2\sigma_3} (\log T)^2 & (\sigma_2=1, \frac{1}{2}<\sigma_1+\sigma_3<1) \\
(\log T)^4 & (\sigma_2=1, \sigma_1+\sigma_3=1) \\
(\log T)^2 &  (\sigma_2<1, \sigma_1+\sigma_2+\sigma_3 = 2) \\
T^{4-2\sigma_1-2\sigma_2-2\sigma_3} & (\sigma_2<1, \frac{3}{2}<\sigma_1+\sigma_2+\sigma_3 < 2).
\end{cases} \nonumber 
\end{align}

Next we estimate $W_{31}$. Similar to the above, we see that
\begin{align}
\label{W_31}
W_{31} &\ll \sum_{m_1 \leq aT} \frac{1}{m_1^{\sigma_1}} \sum_{n_1<m_1} \frac{1}{n_1^{\sigma_2}(m_1+n_1)^{\sigma_3-1}} \\
&\times \sum_{r \leq m_1+n_1} \frac{1}{r(m_1+n_1+r)^{\sigma_1+\sigma_3}} \sum_{n_2 <\frac{r+m_1+n_1}{2}} \frac{1}{n_2^{\sigma_2}} \nonumber \\
&\ll \sum_{m_1 \leq aT} \frac{1}{m_1^{\sigma_1}} \sum_{n_1<m_1} \frac{1}{n_1^{\sigma_2}(m_1+n_1)^{\sigma_3-1}} \nonumber \\
&\times \sum_{r \leq m_1+n_1} \begin{cases} 
\frac{1}{r(m_1+n_1+r)^{\sigma_1+\sigma_3}} & (\frac{3}{2}>\sigma_2>1) \\
\frac{\log (m_1+n_1+r)}{r(m_1+n_1+r)^{\sigma_1+\sigma_3}} & (\sigma_2=1) \\
\frac{1}{r(m_1+n_1+r)^{\sigma_1+\sigma_2+\sigma_3-1}} & (\sigma_2<1) \\
\end{cases} \nonumber \\
&\ll \sum_{m_1 \leq aT} \begin{cases}
\frac{\log m_1}{m_1^{2\sigma_1+2\sigma_3-1}} & (\frac{3}{2}>\sigma_2>1) \\
\frac{(\log m_1)^3}{m_1^{2\sigma_1+2\sigma_3-1}} & (\sigma_2=1) \\
\frac{\log m_1}{m_1^{2\sigma_1+2\sigma_2+2\sigma_3-3}} & (\sigma_2<1) \\
\end{cases} \nonumber \\
&\ll \begin{cases}
T^{2-2\sigma_1-2\sigma_3} \log T& (\frac{3}{2}>\sigma_2>1, \frac{1}{2}<\sigma_1+\sigma_3<1) \\
T^{2-2\sigma_1-2\sigma_3} (\log T)^3 & (\sigma_2=1, \frac{1}{2}<\sigma_1+\sigma_3<1) \\
(\log T)^4 & (\sigma_2=1, \sigma_1+\sigma_3=1) \\
(\log T)^2 &  (\sigma_2<1, \sigma_1+\sigma_2+\sigma_3 = 2) \\
T^{4-2\sigma_1-2\sigma_2-2\sigma_3} \log T & (\sigma_2<1, \frac{3}{2}<\sigma_1+\sigma_2+\sigma_3 < 2).
\end{cases} \nonumber
\end{align}

Therefore by (\ref{W_1 T}), (\ref{W_2}), (\ref{W_32}) and (\ref{W_31}), we have 
\begin{align*}
&\int_2^T \abs{\Sigma_1 (s_1,s_2,s_3)}^2 dt_3 \\ &=\zeta_{AV,2}^{[2]} (s_1,s_2,2\sigma_3)T \\
&+ \begin{cases}
O \left( T^{2-2\sigma_1-2\sigma_3} \log T \right) & (\frac{3}{2}>\sigma_2>1, \frac{1}{2}<\sigma_1+\sigma_3<1) \\
O \left( T^{2-2\sigma_1-2\sigma_3} (\log T)^3 \right) & (\sigma_2=1, \frac{1}{2}<\sigma_1+\sigma_3<1) \\
O \left( (\log T)^4 \right) & (\sigma_2=1, \sigma_1+\sigma_3=1) \\
O \left( (\log T)^2 \right) &  (\sigma_2<1, \sigma_1+\sigma_2+\sigma_3 = 2) \\
O \left( T^{4-2\sigma_1-2\sigma_2-2\sigma_3} \log T \right) & (\sigma_2<1, \frac{3}{2}<\sigma_1+\sigma_2+\sigma_3 < 2). \\
\end{cases} 
\end{align*}
Then
\begin{align*}
&\int_2^T \abs{\zeta_{AV,2} (s_1,s_2,s_3) }^2 dt_3 \\
&= \int_2^T \abs{\Sigma_1 (s_1,s_2,s_3)}^2 dt_3 + O \left( \int_2^T \abs{\Sigma_1(s_1,s_2,s_3)E(s_1,s_2,s_3)} dt_3 \right) \\
&\quad+O \left( \int_2^T \abs{E(s_1,s_2,s_3)}^2 dt_3 \right).  
\end{align*}
Since $\frac{3}{2} < \sigma_1+\sigma_2+\sigma_3 \leq 2$, in this case we have
\begin{align*}
\int_2^T \abs{E(s_1,s_2,s_3)}^2 dt_3 &= O \left( \int_2^T t_3^{3-2\sigma_1-2\sigma_2-2\sigma_3} dt_3 \right) \\
& \ll \begin{cases}
T^{4-2\sigma_1-2\sigma_2-2\sigma_3} & (\frac{3}{2} < \sigma_1+\sigma_2+\sigma_3<2) \\
\log T & (\sigma_1+\sigma_2+\sigma_3=2). \\
\end{cases}
\end{align*}
Thus we see that
\begin{align*}
&\int_2^T \abs{\Sigma_1(s_1,s_2,s_3)E(s_1,s_2,s_3)} dt_3 \\
&\ll \left( \int_2^T \abs{\Sigma_1(s_1,s_2,s_3)}^2 dt_3 \right)^\frac{1}{2} \left( \int_2^T \abs{E(s_1,s_2,s_3)}^2 dt_3 \right)^\frac{1}{2} \\
&\ll \begin{cases} 
T^{\frac{5}{2}-\sigma_1-\sigma_2-\sigma_3} & (\frac{3}{2} < \sigma_1+\sigma_2+\sigma_3<2) \\
(T \log T)^\frac{1}{2} & (\sigma_1+\sigma_2+\sigma_3=2). \\
\end{cases} 
\end{align*}
Comparing the order of the error terms, we obtain
\begin{align*}
&\int_2^T \abs{\zeta_{AV,2} (s_1,s_2,s_3) }^2 dt_3 \\
&=\zeta_{AV,2}^{[2]} (s_1,s_2,2\sigma_3)T \\
&+ \begin{cases}
O ( T^{2-2\sigma_1-2\sigma_3} \log T ) & (\sigma_2 \geq \frac{1}{2}+\sigma_1+\sigma_3) \\
O ( T^{\frac{5}{2}-\sigma_1-\sigma_2-\sigma_3} ) & (\sigma_2 < \frac{1}{2}+\sigma_1+\sigma_3, \frac{3}{2} < \sigma_1+\sigma_2+\sigma_3<2) \\
O ( (T \log T)^\frac{1}{2} ) & (\sigma_2 < \frac{1}{2}+\sigma_1+\sigma_3, \sigma_1+\sigma_2+\sigma_3=2). \\
\end{cases} 
\end{align*}
\end{proof} 

%%%Application
\section{Application}\label{sec4}
By using Theorem \ref{thm:approximation2} and (\ref{functional relation}), we can easily obtain an approximation formula for the Mordell-Tornheim double zeta-function.

\begin{cor}
\label{cor:mt approximation}
For $s_j=\sigma_j+it_j \in \mathbb{C}$ ($j=1,2,3$) with $\sigma_1 \geq 0, t_1 \geq 0$ and $\sigma_2 \geq 0, t_2 \geq 0$.  Assume that $\sigma_3> \max \{ 0, \frac{1}{2}-\sigma_1, \frac{1}{2}-\sigma_2, \frac{3}{2}-\sigma_1-\sigma_2 \}, t_3 \geq2, s_1+s_3 \neq 1, s_2+s_3 \neq 1$ and $s_1+s_2+s_3 \neq 2$. Then we have
\begin{align}
\label{formula:MT appoximation}
\zeta_{MT,2} (s_1,s_2,s_3) &= \sum_{m \leq bt_3} \sum_{n \leq bt_3} \frac{1}{m^{s_1}n^{s_2}(m+n)^{s_3}} \nonumber \\
&\quad+\begin{cases}
O ( t_3^{-\min \{ \sigma_1+\sigma_3, \sigma_2+\sigma_3 \} } ) & (\max \{ \sigma_1, \sigma_2 \} >\frac{3}{2}) \\
O ( t_3^{-\min \{ \sigma_1+\sigma_3, \sigma_2+\sigma_3 \} } \log t_3 ) & (\max \{ \sigma_1, \sigma_2 \}=\frac{3}{2}) \\
O ( t_3^{\frac{3}{2}-\sigma_1-\sigma_2-\sigma_3} ) & (\sigma_1,\sigma_2<\frac{3}{2}) \\
\end{cases}
\end{align}
where $b=\max \{1, t_1, t_2 \}$ and implicit constants depend on $s_1,s_2$ and $\sigma_3$.  
\end{cor}
\begin{proof}
In the proof of Theorem \ref{thm:approximation2}, let $x=y=bt_3$, $\kappa=\frac{4\pi }{3}$ and we replace $a=\max \{1, t_1\}$ with $b=\max \{1, t_1, t_2 \}$, then it holds that 
\[
t_3 \leq \frac{2 \pi x}{\kappa} - t_1, \quad t_3 \leq \frac{2 \pi x}{\kappa} - t_2 
\]
for $t_3 \geq 2$. Thus by Theorem \ref{thm:approximation2}, the following identities hold in the region indicated by the assumptions:
\begin{align*}
\zeta_{AV,2} (s_1,s_2,s_3) &= \sum_{m \leq bt_3} \sum_{n<m} \frac{1}{m^{s_1}n^{s_2}(m+n)^{s_3}} + (\text{error terms $1$}), \\
\zeta_{AV,2} (s_2,s_1,s_3) &= \sum_{m \leq bt_3} \sum_{n<m} \frac{1}{m^{s_2}n^{s_1}(m+n)^{s_3}} + (\text{error terms $2$}), \\
\frac{1}{2^{s_3}} \zeta (s_1+s_2+s_3) &= \sum_{m \leq bt_3} \frac{1}{m^{s_1+s_2}(2m)^{s_3}} +\sum_{m > bt_3} \frac{1}{m^{s_1+s_2}(2m)^{s_3}}.
\end{align*}
where (\text{error terms $1$}) are the all terms except for the first term in the right hand side in (\ref{AV_2}), and (\text{error terms $2$}) are that replaced $s_1$ and $s_2$ in (\text{error terms $1$}). 

In this region, we have 
\begin{align*}
&(\text{error terms $1$}) + (\text{error terms $2$}) \\
&= \begin{cases}
O ( t_3^{-\min \{ \sigma_1+\sigma_3, \sigma_2+\sigma_3 \} } ) & (\max \{ \sigma_1, \sigma_2 \} >\frac{3}{2}) \\
O ( t_3^{-\min \{ \sigma_1+\sigma_3, \sigma_2+\sigma_3 \} } \log t_3 ) & (\max \{ \sigma_1, \sigma_2 \}=\frac{3}{2}) \\
O ( t_3^{\frac{3}{2}-\sigma_1-\sigma_2-\sigma_3} ) & (\sigma_1,\sigma_2<\frac{3}{2}) \\
\end{cases}
\end{align*}
by Theorem \ref{thm:approximation2}.
\end{proof}

Finally, by using Corollary \ref{cor:mt approximation}, we prove Theorem \ref{thm:forth thm}.  
\begin{proof}[Proof of Theorem \ref{thm:forth thm}]
Let  $s_j=\sigma_j+it_j $ ($j=1,2,3$) be in the domain $\mathcal{D}^\prime$. Since $\frac{1}{2} < \sigma_1+\sigma_3 \leq1, \frac{1}{2} < \sigma_2+\sigma_3 \leq1$ and $\sigma_3> 0$, we find that $\sigma_1, \sigma_2<1$. Then by Corollary \ref{cor:mt approximation}, we have
\begin{equation}
\label{mt approximation}
\zeta_{MT,2} (s_1,s_2,s_3) = \sum_{m \leq bt_3} \sum_{n \leq bt_3} \frac{1}{m^{s_1}n^{s_2}(m+n)^{s_3}} +
O ( t_3^{\frac{3}{2}-\sigma_1-\sigma_2-\sigma_3} ).
\end{equation}
Denote the main term of (\ref{mt approximation}) by $\Sigma_2(s_1,s_2,s_3)$ and the error term by $E(s_1,s_2,s_3)$, respectively. Then for $M=\max \{ \frac{m_1}{b},\frac{n_1}{b},\frac{m_2}{b}, \frac{n_2}{b},2\}$,we have
\begin{align*}
&\int_2^T \abs{\Sigma_2 (s_1,s_2,s_3)}^2 dt_3 \\
&= \int_2^T \sum_{m_1,n_1 \leq bt_3} \frac{1}{m_1^{s_1} n_1^{s_2} (m_1+n_1)^{s_3}} \sum_{m_2,n_2 \leq bt_3} \frac{1}{m_2^{\overline{s_1}} n_2^{\overline{s_2}} (m_2+n_2)^{\overline{s_3}}}dt_3 \\
&=  \sum_{\substack{m_1,n_1 \leq bT \\ m_2,n_2 \leq bT}} \frac{1}{m_1^{s_1} m_2^{\overline{s_1}} n_1^{s_2} n_2^{\overline{s_2}} (m_1+n_1)^{\sigma_3} (m_2+n_2)^{\sigma_3}} \int_{M}^T \left( \frac{m_2+n_2}{m_1+n_1} \right)^{it_3} dt_3 \\
&= \sum_{\substack{m_1,n_1,m_2,n_2 \leq bT \\ m_1+n_1=m_2+n_2}} \frac{1}{m_1^{s_1} m_2^{\overline{s_1}}n_1^{s_2}  n_2^{\overline{s_2}} (m_1+n_1)^{2\sigma_3}} (T-M) \\
&+\sum_{\substack{m_1,n_1,m_2,n_2 \leq bT \\ m_1+n_1 \neq m_2+n_2}} \frac{1}{m_1^{s_1} m_2^{\overline{s_1}}n_1^{s_2}  n_2^{\overline{s_2}} (m_1+n_1)^{\sigma_3}(m_1+n_1)^{\sigma_3}} \\
&\times \frac{e^{iT\log \left( \frac{m_2+n_2}{m_1+n_1} \right)}-e^{iM\log \left( \frac{m_2+n_2}{m_1+n_1} \right)}}{i\log \left( \frac{m_2+n_2}{m_1+n_1} \right)} \\
&= S_1 T -S_2+S_3,
\end{align*}
say. Firstly we estimate $S_1$.  
\begin{align}
\label{S_1}
S_1 &= \sum_{\substack{m_1,n_1,m_2,n_2 \leq bT \\ m_1+n_1=m_2+n_2}} \frac{1}{m_1^{s_1} m_2^{\overline{s_1}}n_1^{s_2}  n_2^{\overline{s_2}} (m_1+n_1)^{2\sigma_3}} \\
&= \sum_{\substack{m_1,n_1,m_2,n_2 \geq 1 \\ m_1+n_1=m_2+n_2}} \frac{1}{m_1^{s_1} m_2^{\overline{s_1}}n_1^{s_2}  n_2^{\overline{s_2}} (m_1+n_1)^{2\sigma_3}} \nonumber \\
& - \left( \left( \sum_{\substack{m_1> bT \\ n_1,m_2,n_2 \leq bT \\ m_1+n_1=m_2+n_2}}+\sum_{\substack{n_1> bT \\ m_1,m_2,n_2 \leq bT \\ m_1+n_1=m_2+n_2}}+\sum_{\substack{m_2> bT \\ m_1,n_1,n_2 \leq bT \\ m_1+n_1=m_2+n_2}}+\sum_{\substack{n_2> bT \\ m_1,n_1,m_2 \leq bT \\ m_1+n_1=m_2+n_2}} \right) \right. \nonumber \\
& + \left( \sum_{\substack{m_1,m_2> bT \\ n_1,n_2 \leq bT \\ m_1+n_1=m_2+n_2}}+\sum_{\substack{n_1,n_2> bT \\ m_1,m_2 \leq bT \\ m_1+n_1=m_2+n_2}}+\sum_{\substack{n_1,m_2> bT \\ m_1,n_2 \leq bT \\ m_1+n_1=m_2+n_2}}+\sum_{\substack{m_1,n_2> bT \\ n_1,m_2 \leq bT \\ m_1+n_1=m_2+n_2}} \right) \nonumber \\
& + \left. \left( \sum_{\substack{m_1,n_1,m_2> bT \\ n_2 \leq bT \\ m_1+n_1=m_2+n_2}}+\sum_{\substack{m_1,n_1,n_2> bT \\ m_2 \leq bT \\ m_1+n_1=m_2+n_2}}+ \sum_{\substack{n_1,m_2,n_2> bT \\ m_1 \leq bT \\ m_1+n_1=m_2+n_2}}+\sum_{\substack{m_1,m_2,n_2> bT \\ n_1 \leq bT \\ m_1+n_1=m_2+n_2}} \right) \right) \nonumber \\
&\times \frac{1}{m_1^{s_1} m_2^{\overline{s_1}}n_1^{s_2}  n_2^{\overline{s_2}} (m_1+n_1)^{2\sigma_3}} \nonumber \\
&= \sum_{\substack{m_1,n_1,m_2,n_2 \geq 1 \\ m_1+n_1=m_2+n_2}} \frac{1}{m_1^{s_1} m_2^{\overline{s_1}}n_1^{s_2}  n_2^{\overline{s_2}} (m_1+n_1)^{2\sigma_3}} \nonumber \\
&\quad -\left( \left( U_{11}+U_{12}+U_{13}+U_{14} \right) \right. \nonumber \\
&\qquad+ \left( U_{21}+U_{22}+U_{23}+U_{24} \right) \nonumber \\
&\qquad+ \left. \left( U_{31}+U_{32}+U_{33}+U_{34} \right) \right), \nonumber 
\end{align}
say. The above deformation is given by Miyagawa \cite{Mi} for the purpose of calculating mean values of the Barnes double zeta-function. We have
\[
\sum_{\substack{m_1,n_1,m_2,n_2 \geq 1 \\ m_1+n_1=m_2+n_2}} \frac{1}{m_1^{s_1} m_2^{\overline{s_1}}n_1^{s_2}  n_2^{\overline{s_2}} (m_1+n_1)^{2\sigma_3}} = \zeta_{MT,2}^{[2]} (s_1,s_2,2\sigma_3),
\]
and
\begin{align*}
& \left( U_{11}+U_{12}+U_{13}+U_{14} \right)  \\
&\quad+ \left( U_{21}+U_{22}+U_{23}+U_{24} \right) \\
&\quad+ \left( U_{31}+U_{32}+U_{33}+U_{34} \right) \\ 
&= \left( U_{11}+U_{12}+U_{13}+U_{14} \right) + \left( U_{21}+U_{22}+U_{23}+U_{24} \right) \\
&\quad+ \left( U_{21}+U_{22}+U_{23}+U_{24} \right) + \left( U_{31}+U_{32}+U_{33}+U_{34} \right) \\ 
&\quad- \left( U_{21}+U_{22}+U_{23}+U_{24} \right).
\end{align*}
Since $m_1+n_1=m_2+n_2$, we see that $n_2=m_1+n_1-m_2$ and $m_2 \leq m_1+n_1$, and since $\sigma_1,\sigma_2<1$, we have 
\begin{align*}
U_{11}+U_{21} &= \left( \sum_{\substack{m_1> bT \\ n_1,m_2,n_2 \leq bT \\ m_1+n_1=m_2+n_2}}+\sum_{\substack{m_1,m_2> bT \\ n_1,n_2 \leq bT \\ m_1+n_1=m_2+n_2}} \right) \frac{1}{m_1^{s_1} m_2^{\overline{s_1}}n_1^{s_2}  n_2^{\overline{s_2}} (m_1+n_1)^{2\sigma_3}} \\
&\ll \sum_{\substack{m_1> bT,m_2 \leq m_1+n_1 \\ n_1,n_2 \leq bT \\ m_1+n_1=m_2+n_2}} \frac{1}{m_1^{\sigma_1} m_2^{\sigma_1} n_1^{\sigma_2}  n_2^{\sigma_2} (m_1+n_1)^{2\sigma_3}} \\
&\ll \sum_{\substack{m_1> bT \\ n_1 \leq bT}} \frac{1}{m_1^{\sigma_1} n_1^{\sigma_2}(m_1+n_1)^{2\sigma_3}} \left( \sum_{m_2 <\frac{m_1+n_1}{2}}+\sum_{\frac{m_1+n_1}{2} \leq m_2 <m_1+n_1} \right) \\
&\times \frac{1}{m_2^{\sigma_1}(m_1+n_1-m_2)^{\sigma_2} }\\
&\ll \sum_{\substack{m_1> bT \\ n_1 \leq bT}} \frac{1}{m_1^{\sigma_1} n_1^{\sigma_2}(m_1+n_1)^{\sigma_2+2\sigma_3}}  \sum_{m_2 <\frac{m_1+n_1}{2}} \frac{1}{m_2^{\sigma_1}} \\
&+\sum_{\substack{m_1> bT \\ n_1 \leq bT}} \frac{1}{m_1^{\sigma_1} n_1^{\sigma_2}(m_1+n_1)^{\sigma_1+2\sigma_3}} \sum_{\frac{m_1+n_1}{2} \leq m_2 <m_1+n_1} \frac{1}{(m_1+n_1-m_2)^{\sigma_2}} \\
&\ll \sum_{\substack{m_1> bT \\ n_1 \leq bT}} \frac{1}{m_1^{\sigma_1} n_1^{\sigma_2}(m_1+n_1)^{\sigma_1+\sigma_2+2\sigma_3-1}} \\
&\ll \sum_{m_1> bT} \frac{1}{m_1^{2\sigma_1+\sigma_2+2\sigma_3-1}} \sum_{n_1 \leq bT} \frac{1}{n_1^{\sigma_2}} \\
&\ll \sum_{m_1> bT} \frac{1}{m_1^{2\sigma_1+2\sigma_2+2\sigma_3-2}} \\
&\ll T^{3-2\sigma_1-2\sigma_2-2\sigma_3}.
\end{align*}
By symmetry, we have $U_{12}+U_{22}, U_{13}+U_{23}, U_{14}+U_{24} = O(T^{3-2\sigma_1-2\sigma_2-2\sigma_3})$.
Further by $m_1=m_2+n_2-n_1, n_1 \leq m_1+n_1= m_2+n_2$ and $\sigma_1,\sigma_2<1$, we have
\begin{align*}
U_{21}+U_{31} &= \left( \sum_{\substack{m_1,m_2> bT \\ n_1,n_2 \leq bT \\ m_1+n_1=m_2+n_2}} + \sum_{\substack{m_1,n_1,m_2> bT \\ n_2 \leq bT \\ m_1+n_1=m_2+n_2}}\right) \frac{1}{m_1^{s_1} m_2^{\overline{s_1}}n_1^{s_2}  n_2^{\overline{s_2}} (m_1+n_1)^{2\sigma_3}} \\
&\ll \sum_{\substack{m_1,m_2> bT \\ n_1 \leq m_2+n_2, n_2 \leq bT \\ m_1+n_1=m_2+n_2}} \frac{1}{m_1^{\sigma_1} m_2^{\sigma_1} n_1^{\sigma_2}  n_2^{\sigma_2} (m_1+n_1)^{2\sigma_3}} \\
&\ll T^{3-2\sigma_1-2\sigma_2-2\sigma_3}.
\end{align*}
Again, by symmetry, we have $U_{22}+U_{32}, U_{23}+U_{33}, U_{24}+U_{34} = O( T^{3-2\sigma_1-2\sigma_2-2\sigma_3})$, and $U_{22}, U_{23}, U_{24} = O(T^{3-2\sigma_1-2\sigma_2-2\sigma_3})$.

Therefore we obtain
\begin{equation}
\label{S_1 T}
S_1 T=  \zeta_{MT,2}^{[2]} (s_1,s_2,2\sigma_3) T + O \left( T^{4-2\sigma_1-2\sigma_2-2\sigma_3} \right).
\end{equation}

Next we estimate $S_2$. Since $M \leq m_1+n_1 (=m_2+n_2)$, we can estimate $S_2$ by the same way as $U_{11}+U_{21}$ with $2\sigma_3$ replaced by $2\sigma_3-1$:
\begin{align}
\label{S_2}
S_2 &\ll \sum_{\substack{m_1,n_1,m_2,n_2 \leq bT \\ m_1+n_1=m_2+n_2}} \frac{1}{m_1^{\sigma_1}m_2^{\sigma_1}n_1^{\sigma_2}n_2^{\sigma_2}(m_1+n_1)^{2\sigma_3-1}} \\
&\ll T^{4-2\sigma_1-2\sigma_2-2\sigma_3}. \nonumber
\end{align}

Finally we estimate $S_3$. We have
\begin{align*}
S_3 &\ll \sum_{\substack{m_1,n_1,m_2,n_2 \leq bT \\ m_1+n_1 \neq m_2+n_2}} \frac{1}{m_1^{\sigma_1}m_2^{\sigma_1}n_1^{\sigma_2}n_2^{\sigma_2}(m_1+n_1)^{\sigma_3}(m_2+n_2)^{\sigma_3}} \frac{1}{\log \left( \frac{m_2+n_2}{m_1+n_1} \right) }\nonumber \\
&\ll \sum_{\substack{m_1,n_1,m_2,n_2 \leq bT \\ m_1+n_1 < m_2+n_2 \leq 2(m_1+n_1)}} \frac{1}{m_1^{\sigma_1}m_2^{\sigma_1}n_1^{\sigma_2}n_2^{\sigma_2}(m_1+n_1)^{\sigma_3}(m_2+n_2)^{\sigma_3}} \frac{1}{\log \left( \frac{m_2+n_2}{m_1+n_1} \right) }\nonumber \\
&+ \sum_{\substack{m_1,n_1,m_2,n_2 \leq bT \\ 2(m_1+n_1) < m_2+n_2}} \frac{1}{m_1^{\sigma_1}m_2^{\sigma_1}n_1^{\sigma_2}n_2^{\sigma_2}(m_1+n_1)^{\sigma_3}(m_2+n_2)^{\sigma_3}} \frac{1}{\log \left( \frac{m_2+n_2}{m_1+n_1} \right) }\nonumber \\
&\ll S_{31}+S_{32},
\end{align*}
say. As for $S_{32}$, we have
\begin{align*}
S_{32} &\ll \sum_{m_1,n_1,m_2,n_2 \leq bT} \frac{1}{m_1^{\sigma_1}m_2^{\sigma_1}n_1^{\sigma_2}n_2^{\sigma_2}(m_1+n_1)^{\sigma_3}(m_2+n_2)^{\sigma_3}} \\
&\ll \left( \sum_{m,n \leq bT} \frac{1}{m^{\sigma_1}n^{\sigma_2}(m+n)^{\sigma_3}} \right)^2 \\
&\ll \left( \sum_{m \leq bT} \frac{1}{m^{\sigma_1+\sigma_2+\sigma_3-1}} + \sum_{l \leq bT} \frac{1}{l^{\sigma_1+\sigma_2+\sigma_3}} \right)^2 \\
&\ll \left( T^{2-\sigma_1-\sigma_2-\sigma_3} \right)^2 \\
&\ll T^{4-2\sigma_1-2\sigma_2-2\sigma_3}.
\end{align*}

As for $S_{31}$, setting $r=m_2+n_2-m_1-n_1$, since $n_2 \leq m_2+n_2 =r+m_1+n_1$, we have
\begin{align*}
S_{31} &\ll \sum_{\substack{m_1,n_1,m_2,n_2 \leq bT \\ m_1+n_1 < m_2+n_2 \leq 2(m_1+n_1)}} \frac{1}{m_1^{\sigma_1}m_2^{\sigma_1}n_1^{\sigma_2}n_2^{\sigma_2}(m_1+n_1)^{\sigma_3}(m_2+n_2)^{\sigma_3}} \\
&\times \frac{1}{\log \left( \frac{m_2+n_2}{m_1+n_1} \right)} \nonumber \\
&\ll \sum_{m_1,n_1 \leq bT} \frac{1}{m_1^{\sigma_1}n_1^{\sigma_2}(m_1+n_1)^{\sigma_3}} \sum_{r \leq m_1+n_1} \frac{1}{(m_1+n_1+r)^{\sigma_3}} \frac{1}{\log \left( 1 + \frac{r}{m_1+n_1} \right) } \\
&\times \sum_{n_2 \leq r+m_1+n_1} \frac{1}{(r+m_1+n_1-n_2)^{\sigma_1}n_2^{\sigma_2}} \nonumber \\
&\ll \sum_{m_1,n_1 \leq bT} \frac{1}{m_1^{\sigma_1}n_1^{\sigma_2}(m_1+n_1)^{\sigma_3-1}} \sum_{r \leq m_1+n_1} \frac{1}{r(m_1+n_1+r)^{\sigma_3}} \\
&\times \sum_{n_2 \leq r+m_1+n_1} \frac{1}{(r+m_1+n_1-n_2)^{\sigma_1}n_2^{\sigma_2}} \nonumber \\
&\ll \sum_{m_1,n_1 \leq bT} \frac{1}{m_1^{\sigma_1}n_1^{\sigma_2}(m_1+n_1)^{\sigma_3-1}} \sum_{r \leq m_1+n_1} \frac{1}{r(m_1+n_1+r)^{\sigma_3}} \\
&\times \left( \sum_{n_2 < \frac{r+m_1+n_1}{2}} + \sum_{\frac{r+m_1+n_1}{2} \leq n_2 < r+m_1+n_1} \right) \frac{1}{(r+m_1+n_1-n_2)^{\sigma_1}n_2^{\sigma_2}} \\
&\ll \sum_{m_1,n_1 \leq bT} \frac{1}{m_1^{\sigma_1}n_1^{\sigma_2}(m_1+n_1)^{\sigma_3-1}} \sum_{r \leq m_1+n_1} \frac{1}{r(m_1+n_1+r)^{\sigma_1+\sigma_3}} \\
&\times \sum_{n_2 < \frac{r+m_1+n_1}{2}} \frac{1}{n_2^{\sigma_2}} \\
&+ \sum_{m_1,n_1 \leq bT} \frac{1}{m_1^{\sigma_1}n_1^{\sigma_2}(m_1+n_1)^{\sigma_3-1}} \sum_{r \leq m_1+n_1} \frac{1}{r(m_1+n_1+r)^{\sigma_2+\sigma_3}} \\
&\times \sum_{\frac{r+m_1+n_1}{2} \leq n_2 < r+m_1+n_1} \frac{1}{(r+m_1+n_1-n_2)^{\sigma_1}} \\
&\ll \sum_{m_1,n_1 \leq bT} \frac{1}{m_1^{\sigma_1}n_1^{\sigma_2}(m_1+n_1)^{\sigma_3-1}} \sum_{r \leq m_1+n_1} \frac{1}{r(m_1+n_1+r)^{\sigma_1+\sigma_2+\sigma_3-1}} \\
&\ll \sum_{m_1,n_1 \leq bT} \frac{1}{m_1^{\sigma_1}n_1^{\sigma_2}(m_1+n_1)^{\sigma_1+\sigma_2+2\sigma_3-2}} \sum_{r \leq m_1+n_1} \frac{1}{r} \\
&\ll \sum_{m_1,n_1 \leq bT} \frac{\log (m_1+n_1)}{m_1^{\sigma_1}n_1^{\sigma_2}(m_1+n_1)^{\sigma_1+\sigma_2+2\sigma_3-2}} \\
&\ll \sum_{m_1 \leq bT} \frac{\log m_1}{m_1^{2\sigma_1+2\sigma_2+2\sigma_3-3}} + \sum_{l \leq bT} \frac{\log l}{l^{2\sigma_1+2\sigma_2+2\sigma_3-1}} \\
&\ll T^{4-2\sigma_1-2\sigma_2-2\sigma_3} \log T.
\end{align*}
Then we obtain
\begin{equation}
\label{S_3}
S_3 \ll T^{4-2\sigma_1-2\sigma_2-2\sigma_3} \log T.
\end{equation}

Therefore, by (\ref{S_1 T}), (\ref{S_2}) and (\ref{S_3}), we obtain
\[
\int_2^T \abs{\Sigma_2 (s_1,s_2,s_3)}^2 dt_3 =\zeta_{MT,2}^{[2]} (s_1,s_2,2\sigma_3)T+ O \left( T^{4-2\sigma_1-2\sigma_2-2\sigma_3} \log T \right),
\]
and by (\ref{mt approximation}), we have
\begin{align*}
&\int_2^T \abs{\zeta_{MT,2} (s_1,s_2,s_3) }^2 dt_3 \\
&= \int_2^T \abs{\Sigma_2 (s_1,s_2,s_3) + E(s_1,s_2,s_3)}^2 dt_3 \\
&= \int_2^T \abs{\Sigma_2 (s_1,s_2,s_3)}^2 dt_3 + O \left( \int_2^T \abs{\Sigma_2(s_1,s_2,s_3)E(s_1,s_2,s_3)}dt_3 \right) \\
&+O \left( \int_2^T \abs{E(s_1,s_2,s_3)}^2 dt_3 \right).  
\end{align*}
Here, since $\sigma_1+\sigma_3 \leq 1, \sigma_2+\sigma_3 \leq1$ and $\frac{3}{2} < \sigma_1+\sigma_2+\sigma_3$, we find that $\frac{3}{2} < \sigma_1+\sigma_2+\sigma_3 \leq 2$. Then we have
\begin{align*}
\int_2^T \abs{E(s_1,s_2,s_3)}^2 dt_3 & \ll \begin{cases}
T^{4-2\sigma_1-2\sigma_2-2\sigma_3} & (\frac{3}{2}<\sigma_1+\sigma_2+\sigma_3<2) \\
\log T & (\sigma_1+\sigma_2+\sigma_3=2), \\
\end{cases} 
\end{align*}
and
\begin{align*}
&\int_2^T \abs{\Sigma_2(s_1,s_2,s_3)E(s_1,s_2,s_3)} dt_3 \\
&\ll \begin{cases}
T^{\frac{5}{2}-\sigma_1-\sigma_2-\sigma_3} & (\frac{3}{2}<\sigma_1+\sigma_2+\sigma_3<2) \\
(T \log T)^\frac{1}{2} & (\sigma_1+\sigma_2+\sigma_3=2). \\
\end{cases} 
\end{align*}
Since $4-2\sigma_1-2\sigma_2-2\sigma_3<\frac{5}{2}-\sigma_1-\sigma_2-\sigma_3$, finally we have
\begin{align*}
&\int_2^T \abs{\zeta_{MT,2} (s_1,s_2,s_3) }^2 dt_3 \\
&=\zeta_{MT,2}^{[2]} (s_1,s_2,2\sigma_3)T +\begin{cases}
O ( T^{\frac{5}{2}-\sigma_1-\sigma_2-\sigma_3} ) & (\frac{3}{2}<\sigma_1+\sigma_2+\sigma_3<2) \\
O ( (T \log T)^\frac{1}{2} ) & (\sigma_1+\sigma_2+\sigma_3=2). \\
\end{cases}
\end{align*}
Thus we complete the proof of Theorem \ref{thm:forth thm}.  
\end{proof}

\section*{acknowledgement}
The author would like to thank deeply Dr. Sh\={o}ta Inoue, Dr. Kenta Endo, Dr. Hirotaka Kobayashi and Mr. Kazuma Sakurai for many valuable advice. The author also would like to thank Professor Kohji Matsumoto for his helpful comments.

\end{document}